\documentclass[12 pt,twoside, a4paper]{article}
\usepackage[utf8]{inputenc}
\usepackage[T1]{fontenc}
\usepackage[english]{babel}
\usepackage{amsmath}
\usepackage{amssymb}
\usepackage{multirow}
\usepackage{mathtools}
\usepackage{amsthm}
\usepackage[new]{old-arrows}
\theoremstyle{plain}

\usepackage{xfrac}

\usepackage{amsfonts}
\usepackage{graphicx}
\usepackage[left=2cm,right=2cm,top=2cm,bottom=2cm]{geometry}
\usepackage{hyperref}
\usepackage{pgf}
\usepackage{tikz}
\usepackage{tkz-fct}
\usepackage[all,cmtip]{xy}
\usepackage{tikz-cd}
\usetikzlibrary{intersections,arrows,chains,matrix,positioning,scopes}

\makeatletter
\let\oldfootnote\footnote
\def\footnote{\@ifstar\footnote@star\footnote@nostar}
\def\footnote@star#1{{\let\thefootnote\relax\footnotetext{#1}}}
\def\footnote@nostar{\oldfootnote}
\makeatother

\DeclareMathOperator{\Hom}{Hom}

\DeclareMathOperator{\pr}{pr}

\DeclareMathOperator{\rk}{rk}
\DeclareMathOperator{\com}{com}
\DeclareMathOperator{\Sch}{Sch}
\DeclareMathOperator{\End}{End}

\DeclareMathOperator{\ann}{ann}

\DeclareMathOperator{\Sym}{Sym}
\DeclareMathOperator{\ad}{ad}

\DeclareMathOperator{\Prim}{Prim}

\DeclareMathOperator{\AL}{AL}

\DeclareMathOperator{\im}{im}
\DeclareMathOperator{\id}{id}
\DeclareMathOperator{\imsc}{imsc}
\DeclareMathOperator{\Spec}{Spec}
\DeclareMathOperator{\Set}{Set}

\DeclareMathOperator{\ab}{ab}
\DeclareMathOperator{\res}{res}
\DeclareMathOperator{\GL}{GL}

\DeclareMathOperator{\grade}{grade}
\DeclareMathOperator{\Stab}{Stab}
\DeclareMathOperator{\Irr}{Irr}
\DeclareMathOperator{\projdim}{projdim}
\DeclareMathOperator{\ev}{ev}

\DeclareMathOperator{\coker}{coker}

\DeclareMathOperator{\Lie}{Lie}

\DeclareMathOperator{\car}{char}
\DeclareMathOperator{\Hache}{H}

\DeclareMathOperator{\Frob}{Frob} 

\DeclareMathOperator{\Span}{Span}
\DeclareMathOperator{\Specr}{\boldsymbol{\Spec}}

\newcommand{\spref}[1]{\cite[\href{http://stacks.math.columbia.edu/tag/#1}{Tag~#1}]{SP22}}

\newenvironment{commentaire}[1]{}{}

\makeatletter
\tikzset{join/.code=\tikzset{after node path={%
\ifx\tikzchainprevious\pgfutil@empty\else(\tikzchainprevious)%
edge[every join]#1(\tikzchaincurrent)\fi}}}
\makeatother

\tikzset{>=stealth',every on chain/.append style={join},
         every join/.style={->}}

\everymath{\displaystyle}

\makeatletter

\newcommand*{\relrelbarsep}{.386ex}
\newcommand*{\relrelbar}{%
  \mathrel{%
    \mathpalette\@relrelbar\relrelbarsep
  }%
}
\newcommand*{\@relrelbar}[2]{%
  \raise#2\hbox to 0pt{$\m@th#1\relbar$\hss}%
  \lower#2\hbox{$\m@th#1\relbar$}%
}
\providecommand*{\rightrightarrowsfill@}{%
  \arrowfill@\relrelbar\relrelbar\rightrightarrows
}
\providecommand*{\leftleftarrowsfill@}{%
  \arrowfill@\leftleftarrows\relrelbar\relrelbar
}
\providecommand*{\xrightrightarrows}[2][]{%
  \ext@arrow 0359\rightrightarrowsfill@{#1}{#2}%
}
\providecommand*{\xleftleftarrows}[2][]{%
  \ext@arrow 3095\leftleftarrowsfill@{#1}{#2}%
}
\makeatother

\newcommand{\p}{^{[p]}}
\newcommand*{\cHom}{\mathcal{H}\kern -.5pt om}
\newcommand*{\cEnd}{\mathcal{E}\kern -.5pt nd}

\newcommand{\cro}{[\cdot,\cdot]}

\newcommand{\A}{\mathbb{A}}
\newcommand{\Z}{\mathbb{Z}}
\newcommand{\Q}{\mathbb{Q}}

\newcommand{\C}{\mathbb{C}}

\newcommand{\cF}{\mathcal{F}}
\newcommand{\cK}{\mathcal{K}}

\newcommand{\cL}{\mathcal{L}}
\newcommand{\cZ}{\mathcal{Z}}
\newcommand{\cQ}{\mathcal{Q}}

\newcommand{\E}{\mathcal{E}}

\newcommand{\isor}{ \buildrel\sim\over\longrightarrow}

\newcommand{\rL}{\mathrm L}
\newcommand{\rE}{\mathrm E}
\newcommand{\rp}{\mathrm p}
\newcommand{\rt}{\mathrm t}
\newcommand{\rSL}{\mathrm{SL}}
\newcommand{\rd}{\mathrm d}
\newcommand{\rM}{\mathrm M}
\newcommand{\rU}{\mathrm U}

\def\fab{\mathfrak{ab}}
\def\fg{\mathfrak{g}}
\def\fh{\mathfrak{h}}
\def\fl{\mathfrak{l}}

\def\fz{\mathfrak{z}}
\def\fs{\mathfrak{s}}
\def\fsl{\mathfrak{sl}}
\def\fgl{\mathfrak{gl}}
\newcommand{\fpsl}{\mathfrak{psl}}

\makeatletter
\newcommand*{\defeq}{\mathrel{\rlap{%
                      \raisebox{0.3ex}{$\m@th\cdot$}}%
                      \raisebox{-0.3ex}{$\m@th\cdot$}}%
                      =}
\makeatother

\newtheorem{counter}[subsubsection]{$\!\!$}
\newtheorem{subcounter}[subsection]{$\!\!$}
\newcounter{intro}
\renewcommand{\theintro}{\Alph{intro}}
\newenvironment{definition}{\begin{counter} \rm {\bf Definition.}}{\end{counter}}
\newenvironment{proposition}{\begin{counter} {\bf Proposition.}}{\end{counter}}
\newenvironment{theorem}{\begin{counter} {\bf Theorem.}}{\end{counter}}
\newenvironment{definition*}{\begin{subcounter} \rm {\bf Definition.}}{\end{subcounter}}

\newenvironment{theorem-intro}
{\refstepcounter{intro}\bigskip\noindent {\bf  Theorem~\theintro.}\em}{\rm\bigskip}

\newenvironment{lemma}{\begin{counter} {\bf Lemma.}}{\end{counter}}
\newenvironment{lemma*}{\begin{subcounter} {\bf Lemma.}}{\end{subcounter}}

\newenvironment{corollary}{\begin{counter} {\bf Corollary.}}{\end{counter}}
\newenvironment{corollary*}{\begin{subcounter} {\bf Corollary.}}{\end{subcounter}}
\newenvironment{remark}{\begin{counter} \rm {\bf Remark.}}{\end{counter}}
\newenvironment{remark*}{\begin{counter} \rm {\bf Remark.}}{\end{counter}}

\newenvironment{example}{\begin{counter} \rm {\bf Example.}}{\end{counter}}

\newenvironment{counterexample}{\begin{counter} \rm {\bf Counter-example.}}{\end{counter}}

\newtheorem{theoreme}{Theorem}

\newenvironment{notitle}[1]{\begin{counter} {\bf #1.}\rm}{\end{counter}}

\begin{document}
\begin{center}
    \large \textbf{Moduli of Lie $p$-algebras}
    \normalsize\\
    \bigskip
    \large \textbf{Alice Bouillet}
\end{center}
\bigskip

\footnote*{
}

\footnote*{
}

\small  \noindent \textbf{Abstract.} In this paper, we study moduli spaces of finite-dimensional Lie algebras
with flat center, proving that the forgetful map from Lie $p$-algebras
to Lie algebras is an affine fibration, and we point out a new case of existence of a $p$-mapping. Then we illustrate these results for the special
case of Lie algebras of rank $3$, whose moduli space we build and study over $\mathbb{Z}$.  We extend the classical equivalence of categories between locally free Lie $p$-algebras of finite rank with finite locally free group schemes of height $1$, showing that the centers of these
objects correspond to each other. We finish by analysing the smoothness of the moduli of Lie $p$-algebras of rank $3$,
in particular identifying some smooth components.

\bigskip

\section{Introduction} 
Infinitesimal group schemes over a field $k$ of prime characteristic $p>0$
are those which are finite and connected. Their consideration is crucial
to the study of categories of algebraic groups where they enter the picture
as kernels, intersections or centers and are essential to the good
properties of such categories. They are also of fundamental importance
to the study of individual algebraic groups. Specifically, in the last
decade it has been shown that many aspects of a reductive group $G$ are
reflected by the collection $\{G_r\}$ of its Frobenius kernels, which are
typical examples of infinitesimal group schemes. Providing an exhaustive
survey of these developments with proper attribution would take us too far
away, so we content ourselves with a mention of some of these aspects
with minimal bibliographic indications:
\begin{itemize}
\item[-] geometry: the collection $\{G_r\}$ recovers the universal cover
of $G$, see \cite{Su78};
\item[-] representation theory: the restrictions of suitable simple
$G$-modules provide simple $G_r$-modules, and all simple $G_r$-modules
arise in this way, see \cite{Ja03}, \S~II.3;
\item[-] cohomology: the cohomology of a $G$-module is the inverse limit
of the cohomologies of the associated restricted $G_r$-modules, see
\cite{FP87}.
\end{itemize}
One aspect that is less documented is the moduli theory of infinitesimal
groups. Let us recall that the {\em height} of an infinitesimal group
over a field is the least integer power $h\ge 0$ of Frobenius that
kills it, and introduce the moduli functor $\mathcal{G}_n^h$ of finite flat
infinitesimal group schemes of order $p^n$ and height~$h$ (a variant
would be to consider groups of height bounded by~$h$). Although
Dieudonn\'e theory and its modern successors provide descriptions of
$\mathcal{G}_n^h(R)$ over sufficiently nice
rings~$R$, almost nothing is known over general bases.

Our aim in this article is to study this moduli problem in the case
where $h=1$, and for simplicity we write $\mathcal{G}_n:=\mathcal{G}_n^1$. This is of
course by far the easiest case, because if we write $S\rightarrow
\Spec(\mathbb{F}_p)$ for a base scheme, and $p$-$\mathcal{L}ie_n(S)$ the category
of $n$-dimensional restricted $\mathcal{O}_S$-Lie algebras, then the functor $\Lie$ gives us an
equivalence: 
$$\Lie: \mathcal{G}_n(S)\overset{\thicksim}{\longrightarrow} p\text{-}\mathcal{L}ie_n(S).$$
 
We are thus reduced to studying the moduli of finite-dimensional
Lie algebras and $p$-mappings on them. Our work is divided in two parts:
in the first half of the paper we study the theoretical aspects, and in
the second half we study in detail the three-dimensional case. In the first part we study a Lie algebra $L$ over a scheme $S$, that is, a
vector bundle equipped with a bracket satisfying the Jacobi condition. The
difference of two $p$-mappings on $L$ takes its values in the center $Z(L)$,
which for this reason plays a key role. Our first main result is obtained
after restriction to the flattening stratification $S^*\to S$ of the center,
and is stated as follows.

\begin{theoreme}
Let $L\rightarrow S$ be a Lie algebra vector bundle.
Let us define the functor $X=~X(L)$ of the $p$-mappings on $L$, i.e. $X(T)=\{p$-mappings on $L\times_S T \}$ for all $S$-schemes $T$. Let $\Frob:S\rightarrow S$ be the Frobenius morphism. Then, $X$ is representable by an affine scheme, and is a formally homogeneous space under $\rE~\defeq~\Hom(\Frob_S^*L,Z(L))$.

\noindent Now let us define the restrictable locus of $L$ as follows:
\begin{align*}S^{\res}=S^{\res}(L):
\{S\text{-schemes}\} & \longrightarrow  \Set \\ 
T  &\longmapsto
\begin{cases} 
\{\emptyset\}\text{ if } L_T\text{ is restrictable over }T \\ 
\hspace{0.22cm}\emptyset \text{ otherwise.}
\end{cases}
\end{align*} 
Then if we suppose $Z(L)\rightarrow S$ flat, the following two conditions are verified:
\begin{enumerate}
\item[{\rm 1.}] $S^{\res}$ is representable by a closed subscheme of $S$.
\item[{\rm 2.}] $X\rightarrow S$ factors through $S^{\res}$ and $X\rightarrow S^{\res}$ is an affine space under the vector bundle $\rE\times_S S^{\res}$. 
\end{enumerate}
\end{theoreme}

\noindent It follows in particular that if $Z(L)$ is flat over $S$, then $X \to S^{\res}$ is
smooth. An
interesting question is to know whether this affine space fibration has
global
sections. In general this is not the case. We provide a positive result
under
a condition on the derived Lie algebra $L' \subset L$. Namely, we move to the
next dimension after the case $L'=0$ (where the zero map is a $p$-mapping). See Theorem~\ref{criteresuffisant}.

\begin{theoreme} Let $L\rightarrow S$ be a Lie algebra vector bundle, such that $L'$ is a
locally free subbundle  of rank $1$. We define a map of vector bundles as
follows:
\begin{align*}
   \alpha: L\rightarrow &\End(L')\simeq \mathbb{G}_a\\
    x \mapsto & (\ad(x)_{|L'}) \mapsto \alpha(x).
\end{align*}
Then, the map $L \rightarrow L$, $x\mapsto \alpha(x)^{p-1}x$ is a $p$-mapping on $L$.
Moreover if $L$ is free of rank $2$, this $p$-mapping is unique.
\end{theoreme}

In the rest of the article, we will put our interest on the moduli  stack $\mathcal{L}ie_n$ of
$n$-dimensional Lie algebras, and especially on the case $n=3$. For this, we will introduce the moduli
space $\rL_n$ of \textit{based} Lie algebras locally free of rank $n$, with the natural action of $\GL_n$
on it, by change of basis. We can see that we have the quotient stack
presentation $\mathcal{L}ie_n=[\rL_n/\GL_n]$, so we are led to studying the orbits of the action of
$\GL_n$. You can find the classification on those isomorphism classes in Fulton and Harris' book
\cite{FH91}, but in order to apply our theoretical results and to allow varying
primes~$p$, we reformulate in Subsection \ref{classification} the classification of $3$-dimensional Lie
algebras over algebraically closed fields in a characteristic-free way, giving representatives of the
isomorphism classes defined over $\mathbb{Z}$ and $\mathbb{Z}[T]$. So let $k$ be an algebraically closed
field of characteristic $p>0$. Let us denote by $\thicksim$ the equivalence relation on $k$, 
given by $x \thicksim x'$ if and only if $x'=x$ or $x'=x^{-1}$. Then any Lie algebras of dimension~$3$
over $k$ is isomorphic to exactly one in the following table.

\begin{center}
\renewcommand{\arraystretch}{1.2}
\begin{tabular}{|c|c|c|c|c|c|}
\hline
\multicolumn{2}{|c|}{Name} & Structure & Orbit dimension
& Center dimension & Restrictable \\
\hline
\hline
\multicolumn{2}{|c|}{$\fab_3$} & abelian & 0 & 3 & yes \\
\hline
\multicolumn{2}{|c|}{$\fh_3$} & nilpotent & 3 & 1 & yes \\
\hline
\multicolumn{2}{|c|}{$\mathfrak{r}$} & solvable & 5 & 0 & no \\
\hline
\multicolumn{2}{|c|}{$\fs$} & simple & 6 & 0 &
\begin{tabular}{c|c}
$p\ne 2$ & $p=2$ \\ yes & no
\end{tabular}
\\
\hline
\multirow{4}{*}{$\;\mathfrak{l}_t\;$} &
$\overline{t}\notin \sfrac{\mathbb{F}_p}{\thicksim}$ & solvable & 5 & 0  & no \\
\cline{2-6}
& $\overline{t}\in \sfrac{\mathbb{F}_p}{\thicksim}\!\smallsetminus\!\{\overline{0},\overline{1}\}\!$ & solvable & 5 & 0 & yes\\
\cline{2-6}
& $\overline{t}=\overline{0}$  & solvable & 5 & 1 & yes  \\
\cline{2-6}
& $\overline{t}=\overline{1}$  & solvable & 3 & 0 & yes  \\
\hline
\end{tabular}
\end{center}

Afterward in Subsection \ref{subsection4.2} we supplement the known results by giving more precise
information on the scheme structure of the moduli
space $\rL_3$, that we define over $\Z$. For this, we use liaison theory, as developed by Peskine and Szpiro in \cite{PS74};
in fact $\rL_3$ turns out to be a typical case of a reducible scheme whose
components are linked.

\begin{theoreme}
\begin{enumerate}
\item[{\rm 1)}] The functor $\rL_3$ is representable by an affine flat $\Z$-scheme of finite type.
\item[{\rm 2)}] The scheme $\rL_3$ has two relative irreducible components
$\rL_3^{(1)}$ and $\rL_3^{(2)}$ which are both flat with Cohen-Macaulay
integral geometric fibers of dimension $6$.

\end{enumerate}
\end{theoreme}
For the end, as we said before, we will come back to our equivalence between height one group
schemes and restricted Lie algebras. Because the center of a Lie 
algebra plays a key role in our work, we extend the classical
equivalence of categories between locally free Lie $p$-algebras of
finite rank with finite locally free group schemes of height $1$,
showing that the centers of those objects correspond to each
other in Proposition \ref{centercorres}. For this reason, for $r\leq n$, let us denote by
$p\text{-}\mathcal{L}ie_{n,r}(S)$ the category of
$n$-dimensional restricted $\mathcal{O}_S$-Lie algebras, whose center is locally free of rank $r$, and with the same idea, let us denote by $\mathcal{G}_{n,r}(S)$ the category of finite locally free $S$-group schemes of order $p^n$, of height $1$, whose center is locally free of
rank $p^r$.

\begin{theoreme}\label{TE}
Let $S$ be a scheme of characteristic $p>0$ 
and let $G\rightarrow S$ be a finite locally free group
scheme of height $1$. Let $Z(G)$
denote its center. Then 
$$Z(\Lie(G))=\Lie(Z(G)).$$
Then the classical equivalence of categories 
$$\Lie: \mathcal{G}_n(S)\overset{\thicksim}{\longrightarrow} p\text{-}\mathcal{L}ie_n(S)$$
restricts to an equivalence $$ \Lie:\mathcal{G}_{n,r}(S)\overset{\thicksim}{\longrightarrow} p\text{-}\mathcal{L}ie_{n,r}(S).$$
\end{theoreme}

So using this, we can focus on the object $p\text{-}\mathcal{L}ie_{n,r}(S)$, and because
we have the quotient stack presentation $\mathcal{L}ie_n=[\rL_n/\GL_n]$, we can
focus on $\rL_n$, and especially on $\rL_n^{\res}$ the locally~closed
subscheme of $\rL_n$ where the universal Lie algebra $\mathbb{L}_n\rightarrow \rL_n$ is restrictable. In particular, if $k$ is an algebraically closed field of characteristic
$p>0$, Theorem \ref{TE} and the previous results allow us to count the
centerless finite locally free $k$-group schemes of order $p^3$, of
height $1$. This number is finite, equal to $1$ if $p=2$ and $(p+3)/2$ if
$p\neq 2$ (See Proposition \ref{count}). 

\bigskip
For the end, in the subsections \ref{5.2}, \ref{5.3} and \ref{5.4}, we study the smoothness of the
restrictable locus $\rL_3^{\res}\subset \rL_3$ of $\mathbb{L}_3$ in the different flattening strata of the center. 
For a better understanding of the following theorem, the reader can look at the pictures of Subsection \ref{picture}. 

\begin{theoreme}\label{TD} Let $k$ be an algebraically closed field of characteristic $p>0$. Let $\rL_{3,r}^{\res}\rightarrow \Spec(k)$ be the locally~closed subscheme of $\rL_3$ where the center $Z(\mathbb{L}_3)$ is locally free of rank $r$, and $\mathbb{L}_3$ is restrictable.
\begin{itemize}
    \item[1.] \begin{itemize}
        \item[(i)] If $p\neq 2$, the singular locus of $\rL_{3,0}^{\res}$ is the orbit of $\mathfrak{l}_{-1}$. The singularity remains after intersection with $\rL_3^{(1)}$ but $\rL_{3,0}^{\res}\cap \rL_3^{(2)}$ is smooth.
        \item[(ii)] If $p=2$, the scheme $\rL_{3,0}^{\res}$ is smooth and remains smooth after intersection with any irreducible component.
    \end{itemize} 

    \item[2.] The singular locus of $\rL_{3,1}^{\res}$ is the orbit of $\fh_3$. The singularity remains after intersection with $\rL_3^{(2)}$ but $\rL_{3,1}^{\res}\cap \rL_3^{(1)}$ is smooth.
    \item[3.] The scheme $\rL_{3,2}^{\res}$ is empty.
    \item[4.] The scheme $\rL_{3,3}^{\res}$ is smooth and remains smooth after intersection with any irreducible component.
\end{itemize}
\end{theoreme}
It is well known that in Lie algebra theory, the characteristics $p=2$
and $p=3$ are special. In the previous result we see that the characteristic
$p=2$ appears as a special case, and the reader can see that the case 
$p=3$ needs special care e.g. in the proof of Theorem
\ref{h3singular}.

Thanks to Theorem \ref{TE}, all the assertions of Theorem \ref{TD} hold also for $\mathcal{G}_{3,r}(k)$, i.e.
$\mathcal{G}_{3,r}(k)$ splits in two irreducible components that we denote by
$\mathcal{G}_{3,r}^{(1)}$ and $\mathcal{G}_{3,r}^{(2)}$; and we can say that if
$p\neq 2$, $\mathcal{G}_{3,0}(k)$ is singular, but becomes smooth if we
intersect with $\mathcal{G}_{3,0}^{(2)}$, if $p\neq 2$ it is smooth.
Moreover $\mathcal{G}_{3,1}(k)$ is singular but becomes smooth when we
intersect it with $\mathcal{G}_{3,1}^{(1)}$, $\mathcal{G}_{3,2}(k)$ is empty and
$\mathcal{G}_{3,3}^p(k)$ is smooth. We refer to Corollary \ref{corogroup} for more details. 

\bigskip

\noindent \textbf{Acknowledgements.} For all his ideas, advise, help, I express
warm thanks to my advisor Matthieu Romagny without whom this work would not have
been possible. I would like to thank Marc Chardin for taking
the time to show and explain to me the beautiful liaison theory which helps me
a lot in this article.  For various conversations or help related
to this article, I thank Andrei Benguş-Lasnier, Delphine Boucher, David Bourqui, Marion Jeannin, Bernard Le Stum and Friedrich Wagemann.

I would like to thank the executive and administrative staff
of IRMAR and of the Centre Henri Lebesgue for creating an attractive mathematical environment.

\tableofcontents

\section{Preliminaries on Lie algebras}

\subsection{Definition and theory of Lie \texorpdfstring{$p$}{p}-algebras over a ring}


In this section, we recall basic notations and facts on Lie algebras and
Lie $p$-algebras. We also recall Jacobson's theorems on existence and
uniqueness of $p$-mappings for some Lie algebras over a commutative ring. The
reader can find the proofs for Lie algebras over a field in Strade and Farnsteiner's book on Modular Lie algebras \cite{SF88},
and we verify easily that these proofs do not use the fact that the base ring
is a field.

Let $R$ be a based ring (commutative with unit). An $R$-Lie algebra  is an $R$-module $l$
endowed with an $R$-bilinear alternating map denoted by
$[\cdot,\cdot]:l\otimes_R l\rightarrow l$ satisfying the Jacobi
identity. If $R\rightarrow R'$ is a map of rings, there is an obvious
structure of $R'$-Lie algebra on $l\otimes_RR'$. We denote by
$\End(l)$ the $R$-module of $R$-linear endomorphisms of $l$,
$\ad:l\rightarrow \End(l)$ the map $x \mapsto [x,\cdot]$ and $Z(l)$
the kernel of $\ad$, called the \textit{center} of $l$. If $l$ is
locally free of finite rank as a module, the formation of $\End(l)$
and $\ad$ commutes with base change, but the formation of the center
does not in general. 

Now let us assume that $R$ is an $\mathbb{F}_p$-algebra, and write
$\Frob: R\rightarrow R$ its Frobenius endomorphism.

\begin{commentaire}{
\begin{remark}
Here we ask the bracket of our Lie algebra to be alternating, not only anti-symmetric, because we will be working with Lie algebras in characteristic $2$. That is, we are working with the definition of Humphreys, see \cite{H78}, Chapter, Section~1. 
\end{remark}}
\end{commentaire}


\begin{definition}\label{D3}
We say that a mapping $(\cdot)\p: l\rightarrow l$ is a \textit{$p$-mapping} if:
\begin{trivlist}
\item{(AL1)} for all $x\in l$, $\ad_{x^{[p]}}=(\ad_x)^p$
\item{(AL2)} for all $\lambda \in R$ and $x\in l$, $(\lambda x)^{[p]}=\lambda^px\p$ 
\item{(AL3)} for all $x,y\in l$, $(x+y)^{[p]}=x^{[p]}+y^{[p]} + \sum_{i=1}^{p-1}s_i(x,y)$ $$\text{where  for all } i\text{, } s_i(x,y)\defeq -\frac{1}{i}\sum_u \ad_{u(1)} \ad_{u(2)}\dots \ad_{u(p-1)}(y),$$ and $u$ ranges through the maps from $\{1,\dots,p-1\}$ to $\{x,y\}$ taking $i$ times the value $x$.
\end{trivlist}

\noindent These three conditions are called \textit{Jacobson's identities}.

For instance, we have $$s_1(x,y)=-[y,[y,\dots,[x,y\underbrace{] ]\dots ]}_{p-1} \text{ and } s_{p-1}(x,y)=[x,[x,\dots,[x,y\underbrace{] ]\dots ]}_{p-1}.$$
\end{definition}

\begin{definition}\label{restricted}
A Lie algebra equipped with a $p$-mapping is called \textit{Lie $p$-algebra} or we say that it is \textit{restricted}. If a Lie algebra can be equipped with a $p$-mapping, we say that it is \textit{restrictable}.
\end{definition}

We also recall that a $p$\textit{-morphism} between two Lie
$p$-algebras is a morphism of Lie algebras that commutes with the  
$p$-mappings. A $p$\textit{-ideal} is an ideal stable by the
$p$-mappings. For example, the center $Z(l)$ is always a $p$-ideal,
by the axiom $(\AL1)$.

The next proposition shows that we can endow the image (under a Lie algebra morphism) of a Lie $p$-algebra with a natural $p$-mapping.

\begin{proposition}\label{pmapimage}
Let $(l_1,(\cdot)\p)$ be a Lie $p$-algebra over $R$. Suppose that $f:l_1\rightarrow l_2$ is a Lie algebra morphism such that $\ker(f)$ is a $p$-ideal of $l_1$. Then there exists exactly one $p$-mapping on $f(l_1)$ such that $f:l_1\rightarrow f(l_1)$ is a $p$-morphism.
\end{proposition}

\begin{proof}
See \cite{SF88}, Chapter~2, Section~2.1, Proposition~1.4.
\end{proof}

\begin{theorem} Let $l$ be a Lie algebra over $R$. 
\begin{enumerate}
\item[{\rm 1.}] Let $\gamma_1$ and $\gamma_2$ be two $p$-mappings on~$l$.
Then $\gamma_2-\gamma_1: l\rightarrow Z(l)$ is Frobenius-semi-linear.
\item[{\rm 2.}] Conversely, let $\phi: l\rightarrow Z(l)$ be a
Frobenius-semi-linear map, and $\gamma_1$ a $p$-mapping on $l$.
Then, $\gamma_1+\phi: l\rightarrow l$ is also a $p$-mapping.
\end{enumerate}
\end{theorem}

\begin{proof}
See \cite{SF88}, Chapter~2, Section~2.2, Proposition~2.1.
\end{proof}

The following corollary is a rewording of the previous theorem. It will be
useful for the following sections where we will present results on Lie
$p$-algebras but in a geometric way. 

\begin{corollary} \label{centerlessunique}
Let $l$ be a Lie algebra over $R$. We define $$\rE\defeq
\Hom_{\Frob}(l,Z(l))=\Hom_R(l\otimes_{R,\Frob}R,Z(l))$$ the set of
Frobenius semi-linear maps from $l$ to $Z(l)$ and let $X$ denote the set of
$p$-mappings on~$l$. Then the map:
\begin{align*}
    \rE\times X&\rightarrow X\times X\\
     (\phi,\gamma)& \mapsto (\phi+\gamma,\gamma)
\end{align*}
is bijective. \hfill\(\Box\)

{\rm \noindent In particular, the theorem says that if there exists a $p$-mapping on $l$, it is unique if and only if $\rE=\{0\}$, i.e. if $l$ is locally free of finite rank, the $p$-mapping is unique if and only if $Z(l)=\{0\}$.}
\end{corollary}

The next proposition shows that the hypothesis $(\AL1)$ is essential in the definition of a $p$-mapping, and gives an equivalent condition for a Lie algebra to be restrictable.

\begin{theorem}\label{Jacobson} (\textbf{Jacobson})
Let $l$ be a Lie algebra, free over $R$ with basis $\{x_i\}_{i\in I}$. Let us assume that for all $i \in I$, there exists $y_i\in l$ such that $\ad^p_{x_i}=\ad_{y_i}$. Then, there exists a unique $p$-mapping $(\cdot)^{[p]}:l\rightarrow l$ such that for all $i\in I$, $x_i\p=y_i$.
\end{theorem}

\begin{proof}
You can find the proof in \cite{SF88}, Chapter~2, Section~2.2, Theorem~2.3, but
the initial version is due to Jacobson, in \cite{J62},
Chapter~5, Section~7, Theorem~11.
\end{proof}


\begin{example} (Zassenhaus). You can have a look at \cite{SF88}, Chapter~1, Section~2.7, Theorem~7.9, or at Zassenhaus's article: \cite{Z39} for more details.
Let $l$ be a free Lie algebra over $R$, with Killing form denoted by $B$. We suppose that
$B$ is \textit{non-degenerate}, that is we suppose that the following map 
\begin{align*}
    l&\longrightarrow \Hom_R(l,R)\\
    x & \longmapsto B(x,\cdot)
\end{align*}is an isomorphism.
Then there exists a unique $p$-mapping on $l$.

\end{example}

\subsection{Vector bundles, quotient and image}

In this section, $S$ is a base scheme. We will study vector bundles
equipped with a bracket, in order to study Lie algebras in families. We
start by giving standard definitions and notations about vector bundles.
We use the notation $\mathbb{O}_S$ for the ring scheme $\Spec(\mathcal{O}_S[X])$. 
\begin{definition}\label{N1}
Throughout all this paper, we call a \textit{generalized vector bundle} any scheme which is an $\mathbb{O}_S$-module, isomorphic to an
$\mathbb{O}_S$-module of the form $\mathbb{V(\cF)}\defeq
\Specr(\Sym(\cF))$ with $\cF$ any quasi-coherent $\mathcal{O}_S$-module.\\
We also call \textit{vector bundles} those for which $\cF$ is locally free
of finite rank. In this case, we use the usual covariant equivalence for
which the sheaf of sections of our scheme is $\cF^\vee$.
\end{definition}

\begin{remark*}
Let $F=\Specr(\Sym(\cF^\vee))$ be a vector bundle.
Then $\mathcal{F}$ is the restriction of the functor of points
of $F$ to the small Zariski site of $S$, that is, to the open
subschemes $U\hookrightarrow S$.
\end{remark*}

\begin{definition}\label{D7}
Let $f:E\rightarrow F$ be a morphism of generalized vector bundles over $S$. We define the \textit{kernel} and the \textit{image} of $f$ as the fppf kernel sheaf of $f$ and the fppf image sheaf of $f$, i.e. for all fppf covers $T\rightarrow S$, we have $$\im(f)(T)=\big\{y\in F(T), \exists T'\rightarrow T \text{ fppf covering and } x'\in E(T') \text{ such that } f(x')=y_{|T'}\big\}.$$
\end{definition}

\begin{remark*}
The image is not representable by a scheme in general, but its
formation commutes with base change.

In the following, exact sequences of (generalized) vector bundles
will be understood as exact sequences of fppf sheaves of modules. 
\end{remark*}

\begin{definition}
Let $X\rightarrow S$ be a vector bundle and $Y\hookrightarrow X$ an
$\mathbb{O}_S$-submodule of $X$. We say that $Y$ is a
\textit{subbundle} of $X$ if $Y$ is a vector bundle and $X/Y$ is also a
vector bundle. 
\end{definition}

\begin{remark*}
It is equivalent to be a subbundle of $X$ and to be a locally direct factor of $X$. 
\end{remark*}

\begin{proposition}\label{flatifflocfree}
Let $F\rightarrow S$ be a generalized vector bundle. Let us write
$F=\Specr(\Sym(\cF))$ for a given quasi-coherent $\mathcal{O}_S$-module
$\mathcal{F}$. Then: 
\begin{enumerate}
\item[{\rm 1.}] $F\rightarrow S$ is of finite presentation if and only
if $\cF$ is of finite presentation.
\item[{\rm 2.}] If $\cF$ is of finite presentation, then
 $$ F\rightarrow S \text{ is flat }\Leftrightarrow F\rightarrow S \text{
 is smooth } \Leftrightarrow \cF \text{ is locally free of finite rank.}$$
\end{enumerate}
\end{proposition} 

\begin{proof}
See Görtz and Wedhorn's book \cite{GW20}, Chapter~7, Proposition~7.41.
\end{proof}

For the following, it will be useful to characterize when an
$\mathbb{O}_S$-submodule of a vector bundle is in fact a subbundle. In
order to do this, we establish these two preliminary lemmas.

\begin{lemma}\label{L1}
\begin{enumerate}
\item[{\rm 1.}]
Let $R$ be a Noetherian ring. Then any surjective endomorphism
$\alpha: R \rightarrow R$ is an automorphism.
 \item[{\rm 2.}] Let $R$ be a ring and $\alpha : R' \rightarrow R'$ a surjective $R$-algebra morphism. Then if $R\rightarrow R'$ is of finite presentation, $\alpha$ is an automorphism.
\end{enumerate}
\end{lemma}

\begin{proof}
1. For a contradiction, let us assume that $\alpha$ is not injective: let
$x\in \ker(\alpha), x\neq 0 $ and $n\in \mathbb{N}$. Then $\alpha^n$ is
surjective, so there exists $y \in R$ such that $x=\alpha^n(y)$. Thus,
$\alpha^{n+1}(y)=0$. Then $y\in \ker(\alpha^{n+1})\setminus
\ker(\alpha^n)$. Thus, the sequence $(\ker(\alpha^n))_{n\geq 0}$ is not
stationary, then we get a contradiction. \\
2. Now we suppose that $R$ is any ring and $R\rightarrow R'$ is of finite
presentation. Then by standard arguments, there exists a subring $R_0
\subset R$ of finite type over $\mathbb{Z}$ and an $R_0$-algebra
$R_0\rightarrow R_0'$ of finite presentation such that $R'\simeq
R'_0\otimes_{R_0} R$. Then if $\alpha: R'\rightarrow R'$ is a surjective
$R$-algebra morphism, we can write $\alpha=\alpha_0 \otimes_{R_0} \id_R:
R'_0\otimes R \rightarrow R'_0\otimes R$ where $\alpha_0:R'_0\rightarrow
R'_0$ is surjective. Then thanks to the previous point, $\alpha_0$ is an
automorphism, then so is $\alpha$, as we wanted.
\end{proof}

\begin{lemma} \label{L2}
Let $X\rightarrow S$ be a scheme and $G\rightarrow S$ a flat group scheme
of finite presentation, acting on $X\rightarrow S$. Let $\pi: X\rightarrow
Y$ be a  faithfully flat $S$-morphism of finite presentation and
$G$-invariant. Let us assume that the morphism \begin{align*}
    G\times_S X &\rightarrow X\times_Y X\\
    (g,x)&\mapsto (x,g\cdot x)
\end{align*}
is an isomorphism.
Then, $Y$ is the quotient of $X$ by $G$ in the category of fppf sheaves on $S$.
\end{lemma}

\begin{proof} 
Let $F$ be an fppf sheaf on $S$ and $f:X\rightarrow F$ a $G$-invariant
morphism. 
As $X\rightarrow Y$ is an fppf morphism and  $F$ is an fppf sheaf, the
following sequence is exact:
$$F(Y)\xrightarrow{\pi^*}F(X) \rightrightarrows F(X\times_Y X)$$
and this sequence is isomorphic to this one:
$$F(Y)\overset{\pi^*}{\longrightarrow} F(X)
\xrightrightarrows[proj.]{act.} F(G\times_S X).$$
And this proves the lemma.
\end{proof}

\begin{proposition}\label{subflat}
Let $E\rightarrow S$ be a vector bundle and $F\hookrightarrow E$ an
$\mathbb{O}_S$-submodule of finite presentation. Then $F$ is a subbundle
of $E$ if and only if  $F\rightarrow S$ is flat.
\end{proposition}

\begin{proof}
Let us assume $F$ is a subbundle of $E$. Then by definition, $F\rightarrow
S$ is flat. 
Conversely, let us suppose $F\rightarrow S$ is flat. Then, thanks to
Proposition~\ref{flatifflocfree}, we know that its sheaf of sections is
locally free of finite rank. Then $F$ is a vector bundle. We only need to
show that $E/F$ is also a vector bundle. Let us denote by $\mathcal{E}$
and by $\cF$ the sheaves of sections of  $E$ and $F$. Then $E$ and
$\mathcal{E}$, and $F$ and $\cF$ determine each other. Moreover, for any
$f:S'\rightarrow S$ base change and for any vector bundle $V\rightarrow
S$, we have $(V\times_S S')_{|Zar}=f^*(V_{|Zar})$, and because a
monomorphism of schemes remains a monomorphism after any base change, we
know that the injection $\mathcal{F}\hookrightarrow \mathcal{E}$ remains
injective after any base change. Then the cokernel $\mathcal{Y}$ of this injection is $\mathcal{O}_S$-flat. Because it is also of finite
presentation, it is locally free of finite rank thanks to Proposition
\ref{flatifflocfree}. Let us show now that $Y\defeq
\Specr(\Sym(\mathcal{Y}^\vee))$ is actually the quotient
$E/F$. We have the following exact sequence:
$$0\rightarrow \mathcal{F}\hookrightarrow \mathcal{E}\rightarrow
\mathcal{Y}\rightarrow 0.$$
Dualizing this sequence, we obtain:
$$0\rightarrow \mathcal{Y}^\vee \rightarrow \E^\vee \rightarrow \cF^\vee
\rightarrow 0 .$$
 As $F$ is a subgroup of $E$, it acts on $E$ by left translation. We then
 have the action morphism 
\begin{align*}
    F\times_S E&\rightarrow E\times_S E\\ 
    (f,e)&\mapsto (f+e, e)
\end{align*}
given on the rings by:
\begin{align*}
   \Phi: \Sym(\E^\vee)\otimes_{\Sym(\mathcal{O}_S^\vee)}\Sym(\E^\vee)&\rightarrow \Sym(\cF^\vee)\otimes_{\Sym(\mathcal{O}_S^\vee)} \Sym(\E^\vee)\\
    1\otimes X&\mapsto 1\otimes X \\
    X\otimes 1 &\mapsto \overline{X}\otimes 1 + 1 \otimes X
\end{align*}
for all $X \in \E^\vee=\Sym^1(\mathcal{E}^\vee)$.
Using the definition we see that the elements of the form\\ $X\otimes 1-1
\otimes X$ with $X\in \mathcal{Y}^\vee$ are in the kernel of $\Phi$, then
we obtain a factorized map
$$\tilde{\Phi}:\Sym(\E^\vee)\otimes_{\Sym(\mathcal{Y}^\vee)}\Sym(\E^\vee)
\rightarrow \Sym(\cF^\vee)\otimes_{\Sym(\mathcal{O}_S^\vee)}
\Sym(\E^\vee).$$
Let us show that $\Tilde{\Phi}$ is an isomorphism. First, one can see that
the source and the target of $\tilde{\Phi}$ are sheaves of polynomial
algebras, with the same number of variables, equal to
$\rk(\cF)+\rk(\mathcal{E})$. Moreover, $\tilde{\Phi}$ is surjective
because 
\begin{align*}
    \tilde{\Phi}(1\otimes X)&=1\otimes X \text{ and}\\
    \tilde{\Phi}(X\otimes 1 -1\otimes X)&=\Bar{X}\otimes 1.
\end{align*}
Thus Lemma~\ref{L1} 2. shows that $\tilde{\Phi}$ is an isomorphism. Then
we have an isomorphism $$ F\times_S E\isor E\times_Y E.$$ Hence, using
Lemma~\ref{L2}, we see that $Y$ is the quotient of $E$ by $F$ in the
category of fppf sheaves on $S$, so $E/F=Y=\Specr(\Sym(\mathcal{Y}^\vee))$
is a vector bundle and $F$ is a subbundle of~$E$.
\end{proof}

\begin{proposition}\label{finitepresentation}
Let $E_1$ and $E_2$ be two generalized vector bundles. Let $f: E_1\rightarrow E_2$ be a morphism of generalized vector bundles. If $E_1$ is of finite presentation and if $E_2$ is of finite type, then $\ker(f)$ is of finite presentation.
\end{proposition}

\begin{proof}
By definition, we have:
$$\ker(f)=\Specr(\Sym(\cF_1)\otimes_{\Sym(\cF_2)}\mathcal{O}_S)=\Specr(\Sym(\cF_1)\otimes_{\Sym(\cF_2)}\Sym(\cF_2)/(\cF_2))$$ $$=\Specr(\Sym_{\Sym(\cF_2)}(\cF_1)/f^\#(\cF_2)).$$
Then because $\cF_1$ is of finite presentation and $\cF_2$ is of finite type, $\ker(f)$ is of finite presentation.
\end{proof}

The next statement is a general result about images and kernels
of morphisms of vector bundles, for which we could not find a proof in
the literature. It gives conditions for the kernel and the image
of a vector bundle morphism to be subbundles. For this result,
we first recall that, for any morphism of schemes $f:X\rightarrow Y$,
there exists a smallest closed subscheme of $Y$ that factorizes $f$.
We denote it by $\imsc(f)$ and it is called \textit{the schematic
image of $f$}. See \cite{GW20} Definition and Lemma~10.29. 

\begin{theorem}\label{imagerpz}
Let $f:E\defeq\Specr(\Sym(\E^\vee))\rightarrow F\defeq \Specr(\Sym(\cF^\vee))$ be a morphism of $S$-vector bundles with kernel $K$ and with image $I$. Then the following are equivalent:
\begin{enumerate}
\item[{\rm 1)}] $K\rightarrow S$ is flat.
\item[{\rm 2)}] $I\rightarrow S$ is representable by an $S$-scheme of finite presentation.
\end{enumerate}
Moreover, when these conditions are satisfied, we have:
\begin{enumerate}
\item[{\rm (i)}] $K$ is a subbundle of $E$ and $I$ is a subbundle of $F$. Moreover, the induced morphism $E/K\rightarrow I$ is an isomorphism.
\item[{\rm (ii)}] $I=\imsc(f)$.
\item[{\rm (iii)}] The sheaf of sections of $K$ is $\cK\defeq \ker(\E\rightarrow \cF)$, that of $I$ is $\mathcal{I}\defeq \im(\E\rightarrow \cF)$, and $\E/\cK \simeq \mathcal{I}$. \\
Moreover, the formation of $K,I,$ and  $\cK, \mathcal{I}$ commute with base change.
\end{enumerate}
\end{theorem}

\begin{proof}
$1)\implies 2)$\\
We denote by $\cK$ the sheaf of sections of $K$. Because $E$ and $F$ are both of finite
presentation, Proposition~\ref{finitepresentation} tells us that $K$ is of finite
presentation, then we can apply Proposition~\ref{subflat} to say that $K$ is a subbundle
of $E$. Let us write the following exact sequence:
$$0\rightarrow \cK \rightarrow \E \rightarrow \cQ \rightarrow 0$$ and let us denote $Y\defeq \Specr(\Sym(\mathcal{Q}^\vee))$. 
Doing the same proof as in Proposition~\ref{subflat}, we see that $\mathcal{Q}$ is locally free of finite rank, and $Y$ is the quotient of $E$ by $K$ in the category
of fppf sheaves on $S$. Then $I=Y=E/K$ is representable by an $S$-scheme of finite type, given by~$I=\Specr(\Sym(\mathcal{Q}^\vee))$.
\tikzset{
  closed/.style = {decoration = {markings, mark = at position 0.5 with { \node[transform shape, xscale = .8, yscale=.4] {/}; } }, postaction = {decorate} },
  open/.style = {decoration = {markings, mark = at position 0.5 with { \node[transform shape, scale = .7] {$\circ$}; } }, postaction = {decorate} }} Because $\cQ$ is the cokernel of the injection $\cK\hookrightarrow \E$, we can write $\mathcal{Q} \hookrightarrow \cF$ so we get a surjection $\cF^\vee \twoheadrightarrow \mathcal{Q}^\vee$ hence \hspace*{-2 mm}\begin{tikzcd}[column sep=5mm] I \arrow[hook, closed]{r}{} & F\end{tikzcd} is a closed immersion. But because $I$ factorises $f$, by definition of the schematic image, we have $I\simeq \imsc(f)$. 
\\
$2) \implies 1)$\\
Let us suppose $I$ is representable by an $S$-scheme of finite presentation. In order to prove that $K\rightarrow S$ is flat, it is sufficient to prove that $E\rightarrow I$ is flat. Let $s\in S$. Then $I_s$ is the image of $E_s\rightarrow F_s$ and $K_s$ is its kernel. Because the formation of the kernel and of the image commutes with base change, we have an isomorphism of fppf sheaves
$$E_s/K_s \isor I_s$$ then $E_s\rightarrow I_s$  is flat. Then, using the "critère de platitude par fibres" (see~\cite{EGA4}, troisième partie, théorème 11.3.10), we obtain that $E\rightarrow I$ is flat. Moreover, the morphism $E\rightarrow I$ is surjective in the topological sense because it is surjective as a morphism of fppf sheaves, then $I\rightarrow S$ is flat.

Let us suppose now that that these conditions are satisfied. 
Then looking at the proof of $1) \implies 2)$, we see that $K$ and $E/K$ are vector bundles
on $S$. Using this same proof, we see that~$I$ is also a subbundle of $F$, and that $E/K \simeq I \simeq \imsc(f)$. 
The first part of \rm{(iii)} is true because $K$ is a subbundle of $E$ and $I$ is a subbundle of $F$.
Then for the last assertion, we have to say that the formation of $K$ commutes with base change because it is
a kernel, then because $K$ and $\cK$ determine each other, we see that $\cK$ commutes with base
change. Finally, $I$ commutes with base change because it is a quotient, and then
$\mathcal{I}$ commutes with base change because it is determined by~$I$. 
\end{proof}

\subsection{Lie algebra vector bundles}

In the following, $\cL$ is a Lie $\mathcal{O}_S$-algebra 
locally free of finite rank, whose bracket is denoted by
$\cro$. We denote by $L\defeq
\Specr(\Sym(\cL^\vee))$ the associated vector bundle, and
$[\cdot,\cdot]:L\times L\rightarrow L$ the morphism of schemes we
deduce from the bracket of $\cL$, inducing a Lie $S$-algebra
structure on $L$. We call these kinds of objects \textit{Lie algebra
vector bundles}. 

We denote
by $\cEnd(\cL)$ the $\mathcal{O}_S$-module of
$\mathcal{O}_S$-endomorphisms of $\cL$, and $ad : \cL  \longrightarrow \cEnd(\cL)$. We denote by $\ad:L\rightarrow \End(L)= \Specr(\Sym(\cEnd(\cL)^\vee))$ the corresponding morphism of schemes, and for the end, we denote by $Z(L)\defeq \ker(\ad: L\rightarrow \End(L))$ the \textit{center} of $L$.


\begin{remark*}\label{R3}
By definition, the formation of $\ad$ and $\End(L)$
commutes with base change, and because $\cL$ is
locally free of finite rank, the formation of
$\cEnd(\cL)$ does too.
\end{remark*}

\begin{proposition}\label{C2}
The center $Z(L)$ of a Lie algebra vector bundle is of finite presentation.
\end{proposition}

\begin{proof}
This is straightforward from Proposition~\ref{finitepresentation}.
\end{proof}
Then we see that we are in good conditions for using Proposition~\ref{flatifflocfree} with the center of a Lie algebra vector bundle. 

Moreover, using Theorem~\ref{imagerpz} \rm{(iii)}, we know that if the center $Z(L)\rightarrow S$ is flat, then it is determined by its sheaf of Zariski sections, which is given by $$\mathcal{Z}(\cL)\defeq \ker(ad:\cL\rightarrow \cEnd(\cL)),$$ and so we have $$Z(L)=\Specr(\Sym(\cZ(\cL)^\vee)).$$

\begin{counterexample}
The hypothesis "$Z(L)\rightarrow S$ flat" is essential. Here is a counter-example:
let $R$ be a ring and $L$ be the Lie $R$-algebra with basis $\{x,y\}$ and
bracket defined by $[x,y]=ax$ for some $a\in R$ such that $a \nmid 0$,
that is, geometrically $L=\Spec(\Sym(Rx^*\oplus Ry^*))$.
Let $\Spec(R') \rightarrow \Spec(R)$ be an open immersion. Hence $R'$ is a flat $R$-algebra. Hence $a\nmid 0$ in $R'$. So using the previous notation, we have $\cZ(\cL)=0$.
But whenever $R'$ is a $R/(a)$-algebra, we have $Z(L)(R')=L(R')$. Hence the $Z(L)$ and its sheaf of Zariski sections do not determine each other.
\end{counterexample}

\begin{definition}\label{D13}
Let $\mathcal{L}$ be an $\mathcal{O}_S$-module in Lie algebras. We define its
\textit{derived Lie algebra} as the image sheaf of $[\cdot, \cdot]: \cL \otimes \cL \rightarrow \cL$.

Let $L\rightarrow S$ be a Lie algebra generalised vector bundle. We define its \textit{derived Lie algebra} $L'$ as the fppf image sheaf of $\cro:L\otimes L\rightarrow L$. 
\end{definition}
In general, the derived Lie algebra is not representable. In fact, 
Theorem \ref{imagerpz} tells us that it is representable if and only 
if the kernel of the bracket is flat. Moreover, in this situation, we 
have $L'=\Spec(\Sym(\cL^\vee))$.

\section{The scheme of Lie \texorpdfstring{$p$}{p}-algebra structures}

\subsection{The functor of \texorpdfstring{$p$}{p}-mappings and the restrictable locus} \label{results}

From now on, $S$ is a scheme of characteristic $p>0$, and we globalize the definition of a $p$-mapping from Lie algebras to a definition on Lie algebra vector bundles as follows.

\begin{definition}\label{D8}
Let $L$ be a Lie algebra generalised vector bundle. We say that a morphism of schemes $(\cdot)\p:L\rightarrow L$ is a \textit{$p$-mapping} on $L$ if for all $S$-schemes $T$, it is a $p$-mapping on $L(T)$.
\end{definition} 

\begin{definition}\label{D10}
Let $X\rightarrow S$ be an $S$-scheme, and let $E$ be a generalised vector bundle over $S$. We say  that $X$ is a \textit{formally homogeneous space under $E$} if $E$ acts on $X$ such that the action map
\begin{align*}
    E\times_S X &\rightarrow X\times_S X\\
    (e,x)&\mapsto (e\cdot x,x)
\end{align*}
is a scheme isomorphism. \\ 
Moreover, if $E$ is a vector bundle, we say that $X$ is a \textit{formally affine space under $E$}.
Moreover, if $X\rightarrow S$ has local sections, i.e. if $X\rightarrow S$ is a sheaf epimorphism for the fppf topology, we say that $X$ is an \textit{affine space under $E$}.
\end{definition}

\begin{remark*}\label{R6}
One can show that the second condition is equivalent to have local sections for the étale or for the Zariski topology. This is because $H^1_{\text{fppf}}(S,E)\simeq H^1_{\text{Zar}}(S,E)$. See Milnes's book on étale cohomology \cite{M80}, Chapter~III, §3, Proposition 3.7.
\end{remark*}


\textbf{Notations: }Let $X$ be a scheme of characteristic $p>0$. We
denote by $\Frob_X:X\rightarrow X$ or simply $\Frob$ the absolute
Frobenius morphism of the scheme $X$.

Let $L\rightarrow S$ be a Lie algebra vector bundle. Let us denote by $\rE$ the generalised vector bundle of Frobenius-semilinear morphisms between
$L$ and $Z(L)$: $$\rE\defeq
\Hom_{\Frob}(L,Z(L))=\Hom(\Frob_S^*L,Z(L))=(\Frob_S^*L)^\vee \otimes
Z(L)$$ where the tensor product is taken in the category of vector
bundles over $S$. If $Z(L)$ is a vector bundle, then so is $\rE$.



\begin{theorem}\label{Xrpz}
Let $L\rightarrow S$ be a Lie algebra vector bundle.
Let us define a set-valued functor as follows:
\begin{align*}X:
\{S\text{-schemes}\} & \longrightarrow  \Set \\
T & \longmapsto \big\{p\text{-mappings on }L\times_S T \big\}. 
\end{align*}
Then, $X$ is representable by an affine scheme, and is a formally homogeneous space under $\rE$.  
\end{theorem}

\begin{proof}
Let $\mathcal{L}$ be the Zariski sheaf of sections of $L\rightarrow S$.
Let us show that $X$ is representable. Because the claim is local on the
target, we can suppose $S=\Spec(R)$ affine, small enough so that $\cL$ is free
with basis $x_1,\dots,x_n$ on $S$, i.e. $$\cL=\mathcal{O}_Sx_1\oplus \dots
\oplus \mathcal{O}_Sx_n\text{ , } x_i\in \cL(S)=L(S) \text{ and }
L=\Spec(\mathcal{O}_S[x_1^*,\dots,x_n^*]).$$
Let us define for all $i$ the following morphism:
$$ f_i: S\xrightarrow{x_i}L\xrightarrow{\ad}\End(L)\xrightarrow{\rp} \End(L)$$
where $\rp: \End(L)\rightarrow \End(L)$ maps an endomorphism to its
$p$-power.
Then, by definition of the fiber product, for all $T \rightarrow S$ and $i\in
\{1,\dots,n\}$, we have $$(L\times_{\End(L),(\ad,f_i)} S)(T)=\{y\in L(T),
\ad_y=(\ad_{x_i})^p_{\mid T}\}.$$
Then, by Jacobson's Theorem~\ref{Jacobson}, the map
\begin{align*}
X &\rightarrow (L\times_{\End(L);(\ad,f_1)}S) \times (L\times_{\End(L);(\ad,f_2)} S) \times \dots \times (L\times_{\End(L);(\ad,f_n)} S)\\
\gamma &\mapsto (\gamma(x_1),\dots,\gamma(x_n))
\end{align*}
is an isomorphism. This shows that $X$ is representable.
Let us show now that $X$ is a formally homogeneous space under $\rE$. Let
$T\rightarrow S$ be an $S$-scheme. We can write
$$\rE(T)=\Hom_{\mathbb{G}_{a,T}-mod}(\Frob^*L\times T ,Z(L)\times_S
T)=\Hom_{\mathbb{G}_{a,T}-mod}(\Frob^*L\times_S T,Z(L\times_S T))$$ $$\text{
and }X(T)=\big\{p\text{-structures on }L\times_S T\big\}.$$ 
    Then the morphism 
\begin{align*}
    \rE\times_S X &\rightarrow X\times_S X\\
    (\phi, \gamma)& \mapsto (\phi+\gamma, \gamma)
\end{align*}
is well-defined and is an isomorphism thanks to Corollary~\ref{centerlessunique}.
\end{proof}

\begin{remark}
If we suppose moreover that $Z(L)\rightarrow S$ flat, then $\rE$ is a
vector bundle, so $X$ is a formally affine space under $\rE$.
\end{remark}

For the next theorem, we recall that a scheme $X\rightarrow S$ is said to be
\textit{essentially free} if we can find a cover of $S$ by affine opens $S_i$,
and for all $i$ an affine and faithfully flat $S_i$-scheme $S'_i$, and a cover
$(X'_{i,j})_j$ of $X'_i\defeq X\times_S S'_i$ by affine opens $X'_{ij}$ such
that for all $(i,j)$, the ring of functions of $X'_{ij}$ is a free module on the ring of
$S'_i$.
\begin{theorem}\label{T6} Let $S$ be a scheme. Let $Z\rightarrow S$ be essentially free and let  \tikzset{closed/.style = {decoration =
{markings, mark = at position 0.5 with { \node[transform shape, xscale = .8,
yscale=.4] {/}; } }, postaction = {decorate} }}
\hspace{-2mm}\begin{tikzcd}[column sep=5mm] Y \arrow[hook, closed]{r}{} & Z\end{tikzcd} be a closed subscheme of $Z$. Then, the Weil restriction defined by \begin{align*} \Pi_{Z/S}(Y):
\{S\text{-schemes}\} & \longrightarrow  \Set \\ 
T & \longmapsto   \begin{cases} 
\{\emptyset\}\text{ if } Z_T=Y_T\\ 
\hspace{0.22cm}\emptyset \text{ otherwise}
\end{cases}
\end{align*}
is representable by a closed subscheme of $S$.
\end{theorem}

\begin{proof}
See \cite{SGA3} Tome 2, exposé VIII, Théorème 6.4. 
\end{proof}

\begin{lemma}\label{L3}
Let $$0\rightarrow K \rightarrow E\xrightarrow{\pi}  F \rightarrow 0$$
be an exact sequence of vector bundles (i.e. seen as fppf sheaves) on a scheme $S$. Then, $\pi$ is surjective, Zariski-locally on S. 
\end{lemma}

\begin{proof}
By hypothesis, $E\rightarrow F$ is a $K_F$-torsor for the fppf topology. 
Let $f:S\rightarrow F$ a section on~$F$. Let $E\times_F S$ be the fiber product made with the section $f$. Then by base change, $E\times_F S$ is a $K_F\times_F S$-torsor for the fppf topology. 
But $K_F\times_F S=K$ and
$$\Hache^1_{fppf}(S,K)=\Hache^1_{Zar}(S,K)$$ because $K$ is a vector
bundle over $S$ (see \cite{M80} for more details). 
Then $E\times_F S$ is a $K$-torsor over $S$, for the Zariski topology. 
Then there exist open immersions $g:S'\rightarrow S$ and $h:S'\rightarrow E\times_F S$ such that this diagram commutes:
$$\xymatrix{ & S' \ar[d]^g \ar[ld]_h\\ E\times_F S \ar[d]_{\pr_1} \ar[r] & S \ar[d]^f\\ E \ar[r]_\pi & F }.$$
Then the Zariski section we are looking for is given by
$\pr_1\circ h: S'\rightarrow E$.
\end{proof}

\begin{theorem}\label{Yrpz}
Let $L\rightarrow S$ be a Lie algebra vector bundle whose center $Z(L)\rightarrow S$ is flat. Let us recall the notation $\rE\defeq \Hom_{\Frob}(L,Z(L))$. Let~$X\rightarrow S$ be the functor of $p$-mappings on~$L$ defined above, and let $S^{\res}=S^{\res}(L)$ be defined as: 
\begin{align*}S^{\res}:
\{S\text{-schemes}\} & \longrightarrow  \Set \\ 
T  &\longmapsto
\begin{cases} 
\{\emptyset\}\text{ if } L_T\text{ is Zar-loc. restrictable over }T \\ 
\hspace{0.22cm}\emptyset \text{ otherwise}.
\end{cases} 
\end{align*} 
Then the following two conditions are verified:
\begin{enumerate}
\item[{\rm 1.}] $S^{\res}$ is representable by a closed subscheme of $S$.
\item[{\rm 2.}] $X\rightarrow S$ factors through $S^{\res}$ and $X\rightarrow S^{\res}$ is an affine space under the vector bundle $\rE\times_S S^{\res}$. 
\end{enumerate}
\end{theorem}

\begin{remark}
\begin{itemize}
    \item[-] The functor $S^{\res}$ could have been defined as the unique sub-functor
    of $S$ such that
$$L_T\text{ is Zar-loc. restrictable} \Leftrightarrow T\rightarrow S \text{ can be
factorized by }S^{\res}.$$
Indeed $T\rightarrow S$ can be factorized by $S^{\res}$ if and only if $S^{\res}(T)\neq
\emptyset$ if and only if $L_T$ is Zar-loc. restrictable.
Let $F$ be a subfunctor of $S$ such that $L_T$ is Zar-loc. restrictable if and only if
$T\rightarrow S$ can be factorized by $F$. But for all $T$, $F(T)\subset
\Hom_S(T,S)=\{\ast\}$. Hence by definition, $S^{\res}$ is the only subfunctor of
$S$ satisfying the property above. 
    \item[-] We could have defined $S^{\res}$ to be the locus where a Lie algebra is fppf-loc. restrictable, because this is less restrictive, but the following results will show that those conditions are the same.
    \item[-] By Yoneda, we can see $X(X)=\Hom_S(X,X)\neq \emptyset$
    because $\id\in X(X)$. But by definition, $\id \in X(X)$
    corresponds to a $p$-mapping on $L_{X}$. Then $L_{X}$ is Zar-loc. restrictable
    and we call this mapping the \textit{universal $p$-mapping on $L_{X}$}. 
\end{itemize}
\end{remark}



\begin{proof} 1. Let $I$ be the image of $\ad$. Because $Z(L)$ is flat, $I$ is
a subbundle of $\End(L)$ by Theorem~\ref{imagerpz} \rm{(i)}. Let~$\rho: I
\rightarrow \End(L)$ be the $p$-th power map, restricted to $I$. Let $W=W(L)$ be
the subfunctor of $S$ defined by:

\begin{align*}W:
\{S\text{-schemes}\}& \longrightarrow \Set \\ 
T & \longmapsto 
\begin{cases} 
\{\emptyset\}\text{ if } I_T\text{ is stable by }\rho \\ 
\hspace{0.22cm}\emptyset \text{ otherwise}.
\end{cases}
\end{align*} 

Let us show that $W$ is representable by a closed subscheme of $S$. Let
$T\rightarrow S$ be an $S$-scheme. Then, $I_T$ is $\rho$-stable if and only if
$\rho^{-1}(I_T)\isor I_T$. But \tikzset{closed/.style = {decoration =
{markings, mark = at position 0.5 with { \node[transform shape, xscale = .8,
yscale=.4] {/}; } }, postaction = {decorate} }}
\begin{tikzcd}[column sep=5mm] I \arrow[hook, closed]{r}{} & \End(L)\end{tikzcd} is closed thanks to Theorem
\ref{imagerpz} \rm{(ii)}, and closed immersion are stable by base change. Then
$\rho^{-1}(I)$ is a closed subscheme of $I$. We know that $I\rightarrow S$ is
essentially free because it is a vector bundle, then using Theorem~\ref{T6}
with \tikzset{closed/.style = {decoration = {markings, mark = at position 0.5
with {\node[transform shape, xscale = .8, yscale=.4] {/}; } }, postaction =
{decorate} }}\hspace*{-2mm}\begin{tikzcd}[column sep=5mm] \rho^{-1}(I) \arrow[hook, closed]{r}{} & I\end{tikzcd}, we see that $W$ is a
closed subscheme of $S$.
    
Let us now show that $W=S^{\res}$. First, let us show $S^{\res}\subset W$. Let
$T\rightarrow S$. If $S^{\res}(T)=\emptyset$, there is nothing to prove. Let us
suppose $S^{\res}(T)\neq\emptyset$. Then by definition $L_T$ is Zar-loc. restrictable over
$T$, hence there exists a $p$-mapping on $L_T$, locally on $T$ for the Zariski
topology. We denote this $p$-mapping by $\gamma$. We want to show that $I_T$
is stable by $\rho$. That means we want to show that there exists a map
$\sigma: I_T\rightarrow I_T$ such that the following diagram commutes:
  $$  \xymatrix{ I_T \ar[r]^{\rho_T \hspace{0.5cm}} \ar[rd]_{\sigma}  & \End(L_T)  \\ & I_T.
    \ar@{^{(}->}[u]_i }$$
    Thanks to Theorem~\ref{imagerpz}, we know that $I_T=L_T/Z(L_T)$. But
    $Z(L_T)\subset L_T$ is an ideal of $L_T$, and thanks to $(\AL1)$,
    it is
    stable by any $p$-mapping, so it is stable by $\gamma$. Then $\gamma$
    induces a
    $p$-mapping that we can write $\sigma: L_T/Z(L_T)\rightarrow L_T/Z(L_T)$
    by Proposition~\ref{pmapimage}. If we denote by $\pi: L_T\rightarrow
    L_T/Z(L_T)$ the quotient
    morphism, we have a commutative diagram, where $\rp:\End(L_T)\rightarrow
    \End(L_T)$ is the $p$-power:
    $$\xymatrix{L_T\ar[r]^{\pi \hspace{1.1cm}} \ar[d]^\gamma & L_T/Z(L_T)=I_T
    \ar[d]^\sigma \ar@{^{(}->}[r]^{ \hspace{0.6cm}i} & \End(L_T) \ar[d]^\rp\\
    L_T \ar[r]_{\pi \hspace{1.1cm}} & L_T/Z(L_T)=I_T \ar@{^{(}->}[r]_{
    \hspace{0.6cm}i} &\End(L_T)}$$ 
    The right-hand square of this diagram is commutative thanks to axiom
    $(\AL1)$. 
    As $\rho=p\circ i$, we can calculate $$ \rho \circ \pi = \rp\circ i\circ \pi
    = i \circ \sigma \circ \pi.$$ But $\pi$ is an epimorphism in the category
    of schemes, then we obtain $\rho = i \circ \sigma$, i.e $\rho$ factors via
    $i$ as we wanted. Then, $S^{\res}\subset W$.
   
    Conversely, let $T\rightarrow S$ be such that $I_T$ is stable by $\rho$. Let us show
    that $X(T)$ is nonempty, locally for the Zariski topology on $T$. As everything is
    local on $S$, and $Z_T$ is locally a direct factor in $L_T$, we can assume $T$ is
    affine, small enough such that $L=\Specr(\mathcal{O}_T[x_1^*,\dots,x_n^*])$. Then we
    have the exact sequence
    $$0\rightarrow Z(L) \rightarrow L \xrightarrow{\ad} I \rightarrow 0$$
    and Lemma~\ref{L3} says that $\ad$ is surjective, Zariski locally on $T$. Then, thanks
    to Jacobson's Theorem~\ref{Jacobson}, we know that we have existence of a $p$-mapping,
    Zariski locally on $T$.
    Then $S^{\res}=W$, so $S^{\res}$ is representable by a closed subscheme of $S$.
    
2. Let $T\rightarrow S$ be an $S$-scheme. Let $\gamma \in X(T)$. Then by definition, $L_T$ is restrictable so $S^{\res}(T)=\{\emptyset\}$. Then we define this map: 
\begin{align*}
    X(T)&\rightarrow S^{\res}(T)
    \\\gamma &\mapsto \{\emptyset \} 
\end{align*}
that factorizes $X\rightarrow S$.
    Thanks to Theorem~\ref{Xrpz}, and because $X\times_S X = X\times_{S^{\res}} X$, we know
    that $$\rE_{S^{\res}}\times_{S^{\res}} X\simeq X\times_{S^{\res}} X.$$
    We need to show that $X\rightarrow S^{\res}$ is a sheaf epimorphism for the fppf topology.
    It suffices to show that for all $T\rightarrow S$
    such that $L_T$ is Zar-loc. restrictable, we can find an fppf morphism $T'\rightarrow T$ such that
    there exists a $p$-mapping on $L_{T'}$. We just 
    have to take for $T'$ the Zariski covering on 
    which $L_T$ possesses a $p$-mapping. 
\end{proof}
\begin{corollary}\label{C3}
With the same hypothesis, if $Z(L)=\{0\}$, then $X\simeq S^{\res}$, so $X\rightarrow S$ is a closed immersion.
\end{corollary}

\subsection{A case of existence of a \texorpdfstring{$p$}{p}-mapping}

In general it is not easy to decide if a given
finite-dimensional Lie algebra or a Lie algebra vector bundle admits a $p$-mapping.
Here is a brief review of the easiest cases we have already seen, where such existence is known to hold :
\begin{enumerate}
    \item Associative Lie algebras, with the Frobenius map.
    \item Lie algebras of group schemes.
    \item Lie algebras whose Killing form is nondegenerate (Zassenhaus).
    \item Somewhat opposite to 3. is the abelian case, where $\gamma=0$ is a
$p$-mapping.
\end{enumerate}
The last case corresponds to the situation where the derived Lie algebra
has rank $0$. In the rest of the section, we will extend that case to the
mildly non-abelian case where the derived Lie algebra has rank 1.

Let $E$ be a vector bundle of rank $1$. We remind to the reader that in this
case we have a canonical isomorphism, given on the functor of points by:
\begin{align*}
  \mathbb{G}_a \isor & \End(E) \\
    f\mapsto & \text{ }m_f
\end{align*}
where $m_f$ is the multiplication by $f$.

\begin{theorem}\label{criteresuffisant}
Let $L\rightarrow S$ be a Lie algebra vector bundle, such that $L'$ is a
locally free subbundle  of rank $1$. We define a map of vector bundles as
follows:
\begin{align*}
   \alpha: L\rightarrow &\End(L')\simeq \mathbb{G}_a\\
    x \mapsto & (\ad(x)_{|L'}) \mapsto \alpha(x).
\end{align*}
Then, the map \begin{align*}
    L \rightarrow & L \\
    x\mapsto & \alpha(x)^{p-1}x
\end{align*} is a $p$-mapping on $L$.
Moreover if $L$ is locally free of rank $2$, this $p$-mapping is unique.
\end{theorem}

\begin{proof} 
In order to prove this, we can suppose $L$ is free, such that $L'$ is free, given
by $L'=\mathbb{G}_a\cdot v$. Then the bracket is given on the functor of
points by:
\begin{align*}
    [\cdot,\cdot]: L\times L \rightarrow & L\\
    (x,y) \mapsto & f(x,y)v
\end{align*}
where $f$ is a bilinear alternating form. 
Let us start by showing that for all $T\rightarrow S$, for all~$x~\in~L(T)$,
we have this equality: $f(x,v)=\alpha(x)$.  Let us denote by $\phi:
\mathbb{G}_a\simeq \End(L')$. So we have to show that
$f(x,v)=~\phi^{-1}((\ad(x)_{|L'}))$, i.e. we have to show that
$\phi(f(x,v))=\ad(x)_{|L'}$. Then let $y\in L'(T)$. We can write $y=\lambda v$
with some $\lambda \in \mathbb{G}_a(T)$. Then $\phi(f(x,v))(y)=\lambda f(x,v)
v$ and $(\ad(x)_{|L'})(y)=\lambda f(x,v)v$ so we have the equality we wanted.

Let us show now that the map 
\begin{align*}
   (\cdot)\p : L \rightarrow & L \\
    x\mapsto & \alpha(x)^{p-1}x= f(x,v)^{p-1}v
\end{align*} verifies $(\AL1)$. 
Let $T$ be an $S$-scheme and let $x,y\in L(T)$. Then
$$\ad_{x^{[p]}}(y)=f(f(x,v)^{p-1}x,y)v=f(x,v)^{p-1}f(x,y)v.$$ Moreover, we can
write $$\ad_x(\ad_x(y))=\ad_x(f(x,y)v)=f(x,y)\ad_x(v)=f(x,y)f(x,v)v.$$
Then by induction, we found $$ \ad_x^p(y)=f(x,y)f(x,v)^{p-1}v.$$
Thus, $(\AL1)$ is checked. 

The condition $(\AL2)$ is directly checked by definition. Let us show that our
map respects $(\AL3)$. Let $x,y \in L(T)$. Then,  
\begin{align*}
    (x+y)\p & =f(x+y,v)^{p-1}(x+y)\\
    & =\sum_{k=0}^{p-1}\binom{p-1}{k}f(x,v)^kf(y,v)^{p-1-k}(x+y)\\
    & = \sum_{k=0}^{p-1}\binom{p-1}{k}f(x,v)^kf(y,v)^{p-1-k}x +
    \sum_{k=0}^{p-1}\binom{p-1}{k}f(x,v)^kf(y,v)^{p-1-k}y\\
 & = \sum_{k=0}^{p-2}\binom{p-1}{k}f(x,v)^kf(y,v)^{p-1-k}x+x\p+\sum_{k=1}^{p-1
 }\binom{p-1}{k}f(x,v)^kf(y,v)^{p-1-k}y+y\p\\
    &=x\p+y\p + \sum_{k=1}^{p-1}\binom{p-1}{k-1}f(x,v)^{k-1}f(y,v)^{p-k}x+\sum
    _{k=1}^{p-1}\binom{p-1}{k}f(x,v)^kf(y,v)^{p-1-k}y.
\end{align*}
Let $k\in[1,p-1]$. Let us compute $s_k$.
First, for all $x,y\in l$,
$$\ad_x(\ad_y(v))=\ad_y(\ad_x(v))=f(x,v)f(y,v)v.$$
Let us recall $$s_k(x,y)=-\frac{1}{k}\sum_u \ad_{u(1)} \ad_{u(2)}\dots
\ad_{u(p-1)}(y)$$ where $u$ ranges through the maps $\{1,\dots,p-1\}\rightarrow
\{x,y\}$ taking $k$ times the value $x$. In this sum, all terms corresponding
to a map $u$ such that $u(p-1)=y$ are zero because $\ad_y(y)=0$ and $\ad_{z}$
is linear for all $z\in L(T)$. Then,
\begin{align*}
s_k(x,y)&=-\frac{1}{k}\sum_u \ad_{u(1)} \ad_{u(2)}\dots \ad_{u(p-2)}(f(x,y)v)\\
&=-\frac{f(x,y)}{k}\sum_u \ad_{u(1)} \ad_{u(2)}\dots \ad_{u(p-2)}(v)
\end{align*}
where $u$ ranges through the maps $u:\{1,\dots,p-2\}\rightarrow \{x,y\}$ taking $k-1$
times the value $x$.\\
Then there are $\binom{p-2}{k-1}$ such maps  and, thanks to the previous
remark, they all give the same term in the sum, which is
$\ad_x^{k-1}((\ad_y)^{p-2-(k-1)}(v))$.
Then, \begin{align*}
  s_k&=-\frac{f(x,y)}{k}\binom{p-2}{k-1}f(x,v)^{k-1}f(y,v)^{p-k-1}v
  &=-\frac{1}{k}\binom{p-2}{k-1}f(x,v)^{k-1}f(y,v)^{p-k-1}f(x,y)v.
\end{align*}
But,
$$\frac{p-1}{k}\binom{p-2}{k-1}=\binom{p-1}{k}.$$
Then,
\begin{align*}
  s_k&=\binom{p-1}{k}f(x,v)^{k-1}f(y,v)^{p-k-1}(f(x,v)y-f(y,v)x)\\
  &= \binom{p-1}{k}f(x,v)^{k}f(y,v)^{p-k-1}y - \binom{p-1}{k}f(x,v)^{k-1}f(y,v)^{p-k}x.
\end{align*}
Moreover, $$\binom{p-1}{k-1}+\binom{p-1}{k}=\binom{p}{k}=0 \text{ for all }k \in [1,p-1].$$
Hence,
\begin{align*}
    \sum_{k=1}^{p-1}s_k &=\sum_{k=1}^{p-1}\binom{p-1}{k}f(x,v)^kf(y,v)^{p-1-k}y -\sum_{k=1}^{p-1} \binom{p-1}{k}f(x,v)^{k-1}f(y,v)^{p-k}x\\
     &=\sum_{k=1}^{p-1}\binom{p-1}{k}f(x,v)^kf(y,v)^{p-1-k}y +\sum_{k=1}^{p-1} \binom{p-1}{k-1}f(x,v)^{k-1}f(y,v)^{p-k}x.
\end{align*}

Then the map we set verifies $(\AL1)$, $(\AL2)$ and $(\AL3)$. Finally, this map is a $p$-mapping on our Lie algebra $L$.

Let us now suppose that $L$ is locally free of rank $2$. Let us show that the
$p$-mapping we have defined above is unique. As the functor $X$ of
the $p$-mappings is a sheaf for the Zariski topology, we can suppose
$S=\Spec(R)$ is affine, small enough so that $L$ is free. As we suppose
$\rk(L')=1$, $\rk(L)=2$ and because the bracket is an alternating
bilinear map, we can write for all $x,y \in L(S)$, $$[x,y]=\det(x,y)v$$ for a certain $v\in L(S)$. 

Let us write $v=\begin{pmatrix}v_1 \\v_2
\end{pmatrix} \in L(S)\simeq R^2$ and $x=\begin{pmatrix}x_1 \\x_2 
    \end{pmatrix} \in Z(L(S))$. Then for all $y\in L(S)$,
    $[x,y]=0$. By taking  $y=\begin{pmatrix} 1\\ 0\end{pmatrix}$
    and $y=\begin{pmatrix} 0\\ 1\end{pmatrix}$, we can write: 
    $$\begin{cases} x_1v_1=x_1v_2=0\\
x_2v_1=x_2v_2=0
\end{cases}.$$

But $(v_1,v_2)=R$. Indeed let $q \in \Spec(R)$. Let us suppose $(v_1,v_2)\subset q$. Then
the bracket on $L(T)\otimes R_q/qR_q$ is the zero morphism. This is impossible because the
derived Lie algebra commutes with base change and so has rank $1$ on $R_q/qR_q$. Then
$(v_1,v_2)$ is not contained in any maximal ideal so $(v_1,v_2)=R$. Then there exist
$a,b\in R$ such that $av_1+bv_2=1$ and $x_1=x_1(av_1+bv_2)=0$. With the same arguments, we
obtain $x_2=0$. Thus $Z(L(T))=\{0\}$ and the $p$-mapping is unique thanks to
Corollary~\ref{centerlessunique}.
\end{proof}

\section{The moduli space of Lie \texorpdfstring{$p$}{p}-algebras of rank 3}

In the remaining sections, we illustrate the previous results
in the case of three-dimensional Lie algebras. As stated in the introduction, let us denote
by $\mathcal{L}ie_n$ the moduli stack of $n$-dimensional Lie
algebras, and $\rL_n$ the moduli space of {\em based} Lie
algebras. Then we have the quotient stack presentation
$\mathcal{L}ie_n~=~[\rL_n/\GL_n]$
where $\GL_n$ acts by change of basis, by this action for any $S$-scheme $T$:
\begin{align*}
    \GL_n(T)\times \rL_n(T)&\rightarrow \rL_n(T)\\
    (M, [\cdot,\cdot]_T)&\mapsto [\cdot,\cdot]'_T\defeq \big( v\otimes w\mapsto M^{-1}[Mv,Mw]_T \big).
\end{align*}

Hence we are led to
studying the $\GL_n$-equivariant geometry of $\rL_n$. In the following we will focus on the case $n=3$.
For a fixed prime $p$,
we are interested in the moduli stack $p\text{-}\mathcal{L}ie_3$ of restricted
Lie algebras. For this, we use the morphism $\pi:~p\text{-}\mathcal{L}ie_3\to\mathcal{L}ie_3$
to the moduli stack of three-dimensional Lie algebras. Thanks to Theorem~\ref{Yrpz},
after passing to the flattening stratification of the center of the
universal Lie algebra, the map $\pi$ is an affine bundle, so before studying
$p\text{-}\mathcal{L}ie_3$, we will focus on~$\mathcal{L}ie_3$, i.e. on $\rL_3$. 

For our purposes, it is important to obtain a description available
in all characteristics. Even better, by defining $\rL_3$ as a functor
over $\Z$
and proving its representability we gain insight into its fine
scheme structure and the way the fibers vary. For all this section, we denote by $\rL_{3,k}$ the base change of $\rL_3$ with a field $k$.
Here is a summary of our main results:

\begin{theorem}
\begin{enumerate}
\item[{\rm 1)}] The functor $\rL_3$ is representable by an affine flat $\Z$-scheme of finite type.
\item[{\rm 2)}] The scheme $\rL_3$ has two relative irreducible components
$\rL_3^{(1)}$ and $\rL_3^{(2)}$ which are both flat with Cohen-Macaulay
integral geometric fibers of dimension $6$.
\end{enumerate}
\end{theorem}

In 2) it is noteworthy that the component we call $\rL_3^{(1)}$
is very simple: it is isomorphic to 6-dimensional affine space $\A^6_{\Z}$.
This is crucial because it turns out that the other component~$\rL_3^{(2)}$ is {\em linked} to it in the sense of liaison theory
as developed by Peskine and Szpiro in \cite{PS74}, which provides
powerful tools to deduce its properties.

Here we use the terminology "relative irreducible components" in the sense that $\rL_3^{(1)}$
and $\rL_3^{(2)}$ are flat of finite presentation over $\Z$, and that for all algebraically
closed fields $k$, $\rL_{3,k}^{(1)}$ and $\rL_{3,k}^{(2)}$ are the irreducible components of
$\rL_{3,k}$. We use this terminology because we are over the ring of integers $\Z$ then it
makes more sense. For more details the reader can have a look at \cite{R11}, where the
definition is given in 2.1.1, with a small (but not important for us) difference. Following the
notation of {\em loc. cit.}, we will show in the following
that~$\Irr(\rL_3/\Z)~=~\Spec(\Z)~\amalg~\Spec(\Z)$.

\begin{commentaire}{
In the remaining sections, we illustrate the previous results in the case of
three-dimensional Lie algebras. The moduli stack of such algebras is the
quotient of the space denoted by $\rL_3$ of \textit{based} Lie algebras of rank $3$ by the 
natural action of $\GL_3$ by change of basis, so we are led to studying the
$\GL_3$-equivariant geometry of that space. Using the given basis allows
to identify the moduli space $\rL_3$ with the space of structure constants
of the Lie bracket on the trivial vector bundle.

Then let us suppose $\mathcal{O}_T^3$ has a Lie algebra structure,
written $[\cdot,\cdot]$. Because we suppose the bracket to be bilinear and alternating, we can use the lexicographic order to get $\{[x,y],[x,z],[y,z]\}$ as a basis for
$\mathcal{O}_T^3 \wedge \mathcal{O}_T^3$, and the Lie bracket will be determined by them.

In all this section, we will be using results from \cite{PS74} about liaison theory. Indeed, this theory suits well with our example, and will be really useful to understand the irreducible components of our moduli space. We will see that this component are linked, and because the first one has good property, this other one will have good properties as well.}
\end{commentaire}

\subsection{Classification over an algebraically closed field}\label{classification}

To begin with, we recall the classification of isomorphism classes
of three-dimensional Lie algebras over any algebraically closed
field $k$. That is, the description of the ($\GL_3$-orbits of)
geometric points of the moduli space $\rL_{3,k}$. Historically, the isomorphism classes of complex and real
three-dimensional Lie algebras were classified as early as 1898
in Bianchi's paper \cite{B98}. After the development of the
algebraic theory of Lie algebras, the topic appeared in the
lecture notes of Jacobson's course \cite{J62}. From this moment
the focus shifted to the algebraic variety structure of the set
$\rL_n$ of $n$-dimensional Lie algebras in work of
Vergne~\cite{V66}, Carles~\cite{C79}, Carles and Diakité~\cite{CD84},
Kirillov and Neretin~\cite{KN84}
and others. There, emphasis was put on low dimensions.
Note that this bibliographic selection is by no means complete. Here, in order to allow varying
primes~$p$, we need to reformulate the classification of $3$-dimensional Lie
algebras over algebraically closed fields in a characteristic-free way. 
\bigskip

\begin{notitle}{Some Lie algebras: four discrete ones, and a family}\end{notitle}
We introduce the five Lie algebras involved in the classification
in a way that allows a characteristic-free statement.
Notationally speaking, if $l$ is a Lie algebra over a ring~$R$,
free of rank $3$ with basis $\{x,y,z\}$ and bracket defined by $[x,y]=ax+by+cz$,
$[x,z]=dx+ey+fz$, $[y,z]=gx+hy+iz$ for some coefficients
$a,\dots,i \in R$, then we say that
\begin{center}
\textit{``\,the Lie algebra structure of $l$ is given
by the matrix} $\begin{pmatrix} a&d&g\\ b&e&h\\c&f&i\end{pmatrix}$.''\end{center}
Moreover, for any Lie algebra $l$ and any $v\in l$,
the map $\ad_v$ is linear, so we will always
represent this linear map by its matrix in the base
$\{x,y,z\}$.

The first four Lie algebras are defined over the ring of integers $R=\Z$:
\begin{itemize}
\item[(1)] the {\em abelian Lie algebra} $\fab_3$ with structure
given by the zero matrix,
\item[(2)] the {\em Heisenberg Lie algebra} $\fh_3$ with structure
matrix
$\left(\begin{smallmatrix} 0&0&1\\ 0&0&0 \\ 0&0&0 \end{smallmatrix}\right)$,
\item[(3)] the {\em Lie algebra $\mathfrak{r} $}, with structure matrix
$\left(\begin{smallmatrix} 0&0&0\\ 1&1&0\\0&1&0 \end{smallmatrix}\right)$,
\item[(4)] the {\em simple Lie algebra} $\fs$ with structure matrix
$\left(\begin{smallmatrix} 0&-1&0\\ 0&0&1\\1&0&0 \end{smallmatrix}\right)$.
\end{itemize}
The fifth Lie algebra is a family defined over the polynomial ring $R=\Z[T]$:
\begin{itemize}
\item[(5)] the Lie algebra $\mathfrak{l}_T$ is defined by the structure matrix
$\left(\begin{smallmatrix} 0&0&0\\ 1&0&0\\0&T&0 \end{smallmatrix}\right)$.
\end{itemize}

\begin{notitle}{More about the simple Lie algebra} \end{notitle}
The reader wondering about the place of~$\fsl_2$ and $\fpsl_2$ in the picture will find 
the following explanations useful.
Let us write $\{X,Y,H\}$ and $\{X',Y',Z'\}$ for the classical bases
of $\fsl_2$ and $\fpsl_2$, and $\{x,y,z\}$ for that of $\fs$. We can write a sequence
of morphisms of $\Z$-Lie algebras:
\[
\fsl_2 \overset{\pi}{\longrightarrow} \fpsl_2 \overset{f}{\longrightarrow} \fs \overset{\ad}{\longhookrightarrow} \fgl_3
\]
with $\pi$ and $f$ given by $X\mapsto X' \mapsto 2x$, $Y\mapsto Y'\mapsto y$, $H \mapsto 2Z'\mapsto 2z$. The morphism $f$ is an
isomorphism over $\Z[1/2]$, but a contraction onto the subalgebra
generated by $y$ in  the fiber at the prime $p=2$. 
For any algebraically closed field~$k$, the Lie algebra
$\fs\otimes k$ is the only simple  three-dimensional Lie algebra over $k$ because
in characteristic~$p\ne 2$ we have $\fs\otimes k\simeq \fsl_2\otimes k$,
while if $p=2$ the algebra $\fs\otimes k$ is known as
$W(1,\underline{2})'$,
the derived algebra of the {\em Jacobson-Witt algebra}.
See \cite{SF88}, \S~4.2 for more on $W(n,\underline{m})$,
and especially Strade's paper \cite{S07}, Theorem~3.2 for the case of characteristic~2.
In characteristic $p=2$, the algebra $\fsl_2$ happens
to be isomorphic with the Heisenberg algebra $\fh_3$.
Moreover, again when $2$ is invertible the morphism $\pi$ is an isomorphism so
$\fpsl_2$ is isomorphic to $\fsl_2$ i.e. to $\fs$. But
in characteristic $2$, using the adjoint representation of $\fpsl_2$
in $\fg\fl_3$ we see the bracket is given by the one
denoted by $\mathfrak{l}_1$ is the above classification. The following picture gives a summary
of the situation. Note that in characteristic $2$, the Lie algebra $\mathfrak{sl}_2$ is 
restrictable not simple while the Lie algebra $\mathfrak{s}$ is simple not restrictable.
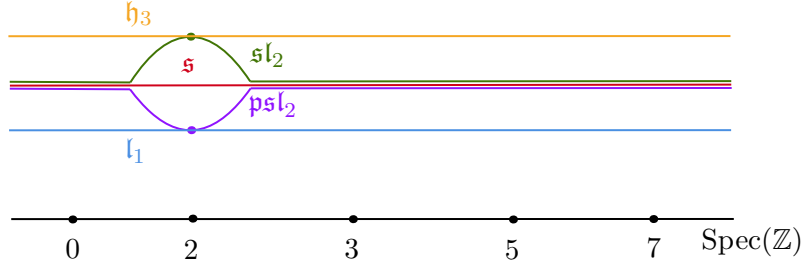
\begin{center}
\tikzset{every picture/.style={line width=0.75pt}} 

\begin{tikzpicture}[x=0.75pt,y=0.75pt,yscale=-1,xscale=1]

\draw [color={rgb, 255:red, 65; green, 117; blue, 5 }  ,draw opacity=1 ]   (159,101.17) -- (219.17,101.5) ;
\draw [color={rgb, 255:red, 65; green, 117; blue, 5 }  ,draw opacity=1 ]   (279,101) -- (519,100.83) ;
\draw [color={rgb, 255:red, 65; green, 117; blue, 5 }  ,draw opacity=1 ]   (219.17,101.5) .. controls (239,70.75) and (259.5,71.25) .. (279,101) ;
\draw [color={rgb, 255:red, 144; green, 19; blue, 254 }  ,draw opacity=1 ]   (219.04,104.96) .. controls (222.47,109.6) and (225.89,113.45) .. (229.29,116.52) .. controls (246.22,131.78) and (262.83,127.57) .. (279.21,104.46) ;
\draw    (159.75,170.25) -- (520.27,170.18) ;
\draw  [color={rgb, 255:red, 0; green, 0; blue, 0 }  ][line width=3] [line join = round][line cap = round] (330.2,170.28) .. controls (330.2,170.28) and (330.2,170.28) .. (330.2,170.28) ;
\draw  [color={rgb, 255:red, 0; green, 0; blue, 0 }  ][line width=3] [line join = round][line cap = round] (190.27,170.43) .. controls (190.27,170.43) and (190.27,170.43) .. (190.27,170.43) ;
\draw  [color={rgb, 255:red, 0; green, 0; blue, 0 }  ][line width=3] [line join = round][line cap = round] (410.09,170.44) .. controls (410.09,170.44) and (410.09,170.44) .. (410.09,170.44) ;
\draw  [color={rgb, 255:red, 0; green, 0; blue, 0 }  ][line width=3] [line join = round][line cap = round] (480.09,170.22) .. controls (480.09,170.22) and (480.09,170.22) .. (480.09,170.22) ;
\draw  [color={rgb, 255:red, 0; green, 0; blue, 0 }  ][line width=3] [line join = round][line cap = round] (250.2,170) .. controls (250.2,170) and (250.2,170) .. (250.2,170) ;
\draw  [color={rgb, 255:red, 65; green, 117; blue, 5 }  ,draw opacity=1 ][line width=3] [line join = round][line cap = round] (249.2,78.47) .. controls (249.2,78.47) and (249.2,78.47) .. (249.2,78.47) ;
\draw [color={rgb, 255:red, 245; green, 166; blue, 35 }  ,draw opacity=1 ]   (158.35,78.25) -- (303,78.22) -- (518.6,78.17) ;
\draw  [color={rgb, 255:red, 144; green, 19; blue, 254 }  ,draw opacity=1 ][line width=3] [line join = round][line cap = round] (249.53,125.52) .. controls (249.53,125.52) and (249.53,125.52) .. (249.53,125.52) ;
\draw [color={rgb, 255:red, 74; green, 144; blue, 226 }  ,draw opacity=1 ]   (158.68,125.58) -- (518.93,125.5) ;
\draw [color={rgb, 255:red, 208; green, 2; blue, 27 }  ,draw opacity=1 ]   (159.09,103.1) -- (519,102.59) ;
\draw [color={rgb, 255:red, 144; green, 19; blue, 254 }  ,draw opacity=1 ]   (159.04,104.63) -- (219.04,104.96) ;
\draw [color={rgb, 255:red, 144; green, 19; blue, 254 }  ,draw opacity=1 ]   (279.21,104.46) -- (473.33,104.32) -- (519.21,104.29) ;

\draw (277.8,79.8) node [anchor=north west][inner sep=0.75pt]  [color={rgb, 255:red, 65; green, 117; blue, 5 }  ,opacity=1 ]  {$\mathfrak{sl}_{2}$};
\draw (243.47,87.73) node [anchor=north west][inner sep=0.75pt]  [color={rgb, 255:red, 208; green, 2; blue, 27 }  ,opacity=1 ]  {$\mathfrak{s}$};
\draw (502.67,172.73) node [anchor=north west][inner sep=0.75pt]    {$\Spec(\mathbb{Z})$};
\draw (185.27,178.8) node [anchor=north west][inner sep=0.75pt]    {$0$};
\draw (244.4,179.6) node [anchor=north west][inner sep=0.75pt]    {$2$};
\draw (324.6,178.8) node [anchor=north west][inner sep=0.75pt]    {$3$};
\draw (404.6,179.4) node [anchor=north west][inner sep=0.75pt]    {$5$};
\draw (215,59) node [anchor=north west][inner sep=0.75pt]  [color={rgb, 255:red, 65; green, 117; blue, 5 }  ,opacity=1 ]  {$\textcolor[rgb]{0.96,0.65,0.14}{\mathfrak{h}_{3}}$};
\draw (475.27,178.4) node [anchor=north west][inner sep=0.75pt]    {$7$};
\draw (212.57,134.05) node [anchor=north west][inner sep=0.75pt]   [align=left] {};
\draw (289.27,60.67) node   [align=left] {\begin{minipage}[lt]{115.6pt}\setlength\topsep{0pt}

\end{minipage}};
\draw (215,128) node [anchor=north west][inner sep=0.75pt]    {$\textcolor[rgb]{0.29,0.56,0.89}{\mathfrak{l}_{1}}$};
\draw (276.8,104.13) node [anchor=north west][inner sep=0.75pt]  [color={rgb, 255:red, 65; green, 117; blue, 5 }  ,opacity=1 ]  {$\textcolor[rgb]{0.56,0.07,1}{\mathfrak{psl}_{2}}$};

\end{tikzpicture}

\end{center}
It can be surprising to find that the group $\rU_3$ of upper-triangular unipotent matrices of size~$3$ and the reductive group $\rSL_2$ have the same Lie algebra in
    characteristic $2$. But those Lie algebras are not isomorphic
    as restricted Lie algebras. Indeed, let $k$ be a field of characteristic $2$. Seeing $\rU_3$ as a
    subgroup
    of $\GL_3$, we can see $\fh_3$ as a $2$-subalgebra of $\rM_3(k)$, where the $2$-mapping on $\rM_3(k)$ is the square map. Then we obtain that the $2$-mapping on $\fh_3$
    is $\gamma \equiv 0$. Doing the same for $\rSL_2$, i.e. seeing
    the group $\rSL_2\subset \GL_2$, we obtain the $2$-mapping:
    $x,y\mapsto 0$ and $z\mapsto z$.
\newline

\begin{notitle}{More about the family $\mathfrak{l}_T$} 
The Lie algebra $\mathfrak{l}_0$ has center of dimension $1$
and a $1$-dimensional derived Lie algebra
$\fg'_0=\Span(y)$.
Now let us suppose $t\in k$ for some field $k$ and
$t\neq 0$. Then the Lie algebra~$\mathfrak{l}_t$ has a
trivial center and 2-dimensional derived Lie algebra
$\fg'_t=\Span(y,z)$. The adjoint action
$\ad:\mathfrak{l}_t\to\End(\fg'_t)$ factors through
$\mathfrak{l}_t^{\ab}:=\mathfrak{l}_t/\fg'_t$ which is free of rank 1.
Any generator of $\mathfrak{l}_t^{\ab}$ is of the form $ux$
for some unit $u\in k^\times$ and acts on $\fg'_t$
with eigenvalues $\{u,ut\}$. We see that the ratio of
eigenvalues is well-defined up to inversion: that is,
the class of $t$ modulo
the equivalence relation $t\sim t^{-1}$ is
independent of $u$
and thus intrinsic to $\mathfrak{l}_t$. In this way we see that for
every field $k$ and elements $t,t'\in k^\times$ we have:
$\mathfrak{l}_t\simeq\mathfrak{l}_{t'}$ if and only if $t'\in \{t,t^{-1}\}$.
 \end{notitle}
Here is the main theorem of this subsection:

\begin{theorem}\label{touteslesinfos}
Let $k$ be an algebraically closed field and denote by $p$
its characteristic. Then any Lie algebra of dimension~$3$
over $k$ is isomorphic to exactly one in the following table.

\begin{center}
\renewcommand{\arraystretch}{1.2}
\begin{tabular}{|c|c|c|c|c|c|}
\hline
\multicolumn{2}{|c|}{Name} & Structure & Orbit dimension
& Center dimension & Restrictable \\
\hline
\hline
\multicolumn{2}{|c|}{$\fab_3$} & abelian & 0 & 3 & yes \\
\hline
\multicolumn{2}{|c|}{$\fh_3$} & nilpotent & 3 & 1 & yes \\
\hline
\multicolumn{2}{|c|}{$\mathfrak{r}$} & solvable & 5 & 0 & no \\
\hline
\multicolumn{2}{|c|}{$\fs$} & simple & 6 & 0 &
\begin{tabular}{c|c}
$p\ne 2$ & $p=2$ \\ yes & no
\end{tabular}
\\
\hline
\multirow{4}{*}{$\;\mathfrak{l}_t\;$} &
$\overline{t}\notin \sfrac{\mathbb{F}_p}{\thicksim}$ & solvable & 5 & 0  & no \\
\cline{2-6}
& $\overline{t}\in\sfrac{\mathbb{F}_p}{\thicksim}\!\smallsetminus\!\{\overline{0},\overline{1}\}\!$ & solvable & 5 & 0 & yes\\
\cline{2-6}
& $\overline{t}=\overline{0}$  & solvable & 5 & 1 & yes  \\
\cline{2-6}
& $\overline{t}=\overline{1}$  & solvable & 3 & 0 & yes  \\
\hline
\end{tabular}
\end{center}
\end{theorem}

\color{black}
\begin{commentaire}{
As $\GL_3$ is a smooth group scheme, the orbits are also smooth so they are determined by
their underlying topological set. And because everything if of finite type, we will just
describe the \textit{geometric points} of this action. As said before the
classification of the orbits of three-dimensional Lie algebras over $\C$ has
already been done, for example in \cite{J62} or \cite{FH91}. Here we give a uniform description in any characteristic, and we see that we
can always take a representative of the orbits which can be defined over $\Z$,
except for the Lie algebras denoted by $\mathfrak{l}_T$, which
are defined over the polynomial ring $\Z[T]$.  }
\end{commentaire}

\bigskip
We split the proof in three parts: first of all we list the isomorphism
classes (\ref{isoclasses}), then we compute the dimensions of the orbits (\ref{dimorbit}) and
finally we determine the restrictable Lie algebras (\ref{restrictableorbit}). 
\newline

\begin{notitle}{Proof of the statement on isomorphism classes}\label{isoclasses} 
In order to have the list of the different orbits, we are following the
proof in Fulton and Harris's book \cite{FH91}, Chapter~10. In this chapter the
proof is divided in three parts, depending on the dimension of the derived Lie algebra. In
this book though, the classification is done over the ring of complex
numbers. The reader can verify that the proof can be
generalised to any field of characteristic~$\neq 2$, and 
up to a change of basis for the Lie algebra whose 
Lie structure is given by the matrix $$\begin{pmatrix}
0&-2&0\\
0&0&2\\
1&0&0 \end{pmatrix},$$ we find the classification we
claim in the theorem (indeed changing $X$ into $2X$ and $H$ into $2H$,
we find the Lie algebra $\fs$).

\noindent Now let us suppose $\car(k)=2$. The reader
can verify that the proof done in \cite{FH91} can still
be generalised until the {\em loc. cit.} $\S 10.4$,
where the authors consider Lie algebras with derived
Lie algebra of rank $3$. Indeed, in this part,
they use an argument that is no longer true in 
characteristic~$2$: a certain endomorphism denoted by
$\ad_H$ has three eigenvalues: $0, \alpha$ and
$-\alpha$, and because $\alpha\neq 0$, these three
eigenvalues are different, then this endomorphism is
diagonalizable. So now let us transform this argument
in our case. So let $\mathfrak{g}$ be a Lie algebra
over $k$, with derived Lie algebra of rank $3$. Let
us do the same proof as done in {\em loc. cit.} $\S
10.4$ until this argument. Then, changing the
eigenvector $X$ of $\ad_H$ for the eigenvalue
$\alpha$ into $\alpha X$, and changing $H$ in
$\alpha^{-1} H$, we found that $\ad_H$ has $0$ and
$1$ as eigenvalues. If $\ad_H$ is diagonalizable, we
can apply the proof of \cite{FH91}. Otherwise,
we can apply the Jordan–Chevalley decomposition to
$\ad_H$ and so we can suppose there is a basis
$\{X,Y,H\}$ of $\mathfrak{g}$ such
that $[H,X]=X$ and $[H,Y]=X+Y$. Then thanks to the
Jacobi condition, we know that
$$[H,[X,Y]]=[X,[H,Y]]+[Y,[H,X]]=[X,Y]+[X,Y]=0.$$
Then $[X,Y]=\beta H$ with $\beta\neq 0$ because the
derived Lie algebra of $\mathfrak{g}$ is of dimension
$3$. Changing $X$ into $aX$ and $Y$ into $aY$ where
$a^2=\beta^{-1}$, we can suppose $\beta=1$ and using
the matrix notation, we can suppose the bracket of
$\fg$ is given by $$\begin{pmatrix}
0&1&1\\
0&0&1\\
1&0&0 \end{pmatrix}$$ in the basis $\{X,Y,H\}$. Using the basis $x=X,
y=X+Y+H$ and $z=X+H$, we obtain
$$    \begin{cases}
        [x,y]=[X,Y]+[X,H]=H+X=z \\
        [x,z]=[X,H]=X=x\\
        [y,z]=[X,H]+[Y,X]+[Y,H]+[H,X]=[X,Y]+[Y,H]=H+X+Y=y.
\end{cases}$$

Hence we finally find the Lie algebra structure of $\fs$, so we find our classification. 
$\square$
\end{notitle}
\begin{remark}
Here we use the terminology of \cite{FH91} for the
Lie algebras $\mathfrak{l}_t$, in particular for the Lie algebra
$\mathfrak{l}_{-1}$. Actually you can find in the literature (for
example in \cite{KN84}) the terminology
$\mathrm{m}(2)$ for this one. This name is due to the
fact it is the Lie algebra of the group  $\rM(2)$ of
euclidean motions of the plane.
\end{remark}

\bigskip

\begin{notitle}{Proof of the statement on the dimension of the orbits}\label{dimorbit}
\end{notitle}
\noindent From now on, we use the notation $o(l)$ for the orbit of a Lie algebra $l$ under the group $\GL_3$.
In order to find the dimension of the orbits, we can calculate
the dimension of the stabilizer, and use the orbit-stabilizer relation. Let $l$ be a Lie algebra over $k$, i.e. $l \in \rL_{3,k}(k)$. Then the orbit of $l$ is the image of this $k$-morphism:  
\begin{align*}
    \GL_3(k)\longrightarrow \rL_{3,k}(k)\\
    A\mapsto A\cdot l.
\end{align*}

Let $A=\begin{pmatrix}
    a_{1,1}&a_{1,2}&a_{1,3}\\a_{2,1}&a_{2,2}&a_{2,3}\\a_{3,1}&a_{3,2}&a_{3,3}\end{pmatrix}\in \GL_3(k)$ be a matrix in the
    stabilizer of $o(l)$. Then, we write $$[Av,Aw]=A[v,w]$$ for
    the elements of the basis and we can fin the equations for
    the stabilizer. For example let us fix a $t$ in some field $k$ and let us do it for $\mathfrak{l}_t$. We
    obtain theses conditions:
$$\left\{\begin{array}{lll}
        (a_{1,1}a_{2,2}-a_{2,1}a_{1,2})y+t(a_{1,1}a_{3,2}-a_{3,1}a_{1,2})z&=
        a_{1,2}x+a_{2,2}y+a_{3,2}z, \\
        (a_{1,1}a_{2,3}-a_{1,3}a_{2,1})y+t(a_{1,1}a_{3,3}-a_{3,1}a_{1,3})z&
        =t a_{1,3} x+t a_{2,3} y+t a_{3,3} z\\
        (a_{1,2}a_{2,3}-a_{2,2}a_{1,3})y+t(a_{1,2}a_{3,3}-a_{1,3}a_{3,2})z&=0
 \end{array}
 \right.$$

Then for instance if $t=0$, the conditions of the stabilizer are now:\\
$$\left\{
    \begin{array}{lll} a_{1,2}=a_{3,2}=0\text{, } a_{1,1}a_{2,2}=a_{2,2}\\  
    a_{1,1}a_{2,3}-a_{1,3}a_{2,1}=0\\
    a_{2,2} a_{1,3} =0
\end{array}.
\right.$$

But $\det(A)=a_{2,2}a_{3,3}\neq 0$ then $a_{1,3}=0$ and $a_{1,1}=1$, so $a_{2,3}=0$. Hence

$$\Stab(\mathfrak{l}_0)=\left\{A \in \GL_3(k), A=\begin{pmatrix}
    1&0&0\\a_{2,1}&a_{2,2}&0\\a_{3,1}&0&a_{3,3}\end{pmatrix}\right\}.$$
    Then $\dim(\Stab(\mathfrak{l}_0))=4$, so $\dim(o(\mathfrak{l}_0))=5.$
Now let us suppose $t \neq 0$ and $t \neq 1$. Doing the same type of
calculation, we obtain again: 
    $$\Stab(\mathfrak{l}_t)=\left\{A \in \GL_3(k), A=\begin{pmatrix}
    1&0&0\\a_{2,1}&a_{2,2}&0\\a_{3,1}&0&a_{3,3}\end{pmatrix}\right\}.$$
    Then $\dim(\Stab(\mathfrak{l}_t))=4$ and $\dim(o(\mathfrak{l}_t))=5.$

We can do the same calculations for the other orbits in order to
find the announced dimensions. The details are left to the reader. \hfill $\square$

\bigskip

\begin{notitle}{Proof of the statement on the restricted orbits}\label{restrictableorbit} 
Now we can have a look at the restrictable orbits. 
Let us suppose for this section that~$\car(k)=p>0$.
\end{notitle}
\begin{enumerate}
\item On the abelian Lie algebra, $\gamma\equiv0$ is a $p$-mapping. 
    \item  The Lie algebra $\fh_3=\Lie(\rU_3)$ is algebraic, hence restrictable.
    \item The Lie algebra $\fs$ is restrictable if $\car(k)\neq 2$, because then $\fs\simeq \fsl_2$ so it is algebraic. But if $\car(k)=2$, $\fs$ is not restrictable: one can see that $\ad_x^2$ is not a linear combination of $\ad_x$, $\ad_y$ and $\ad_z$, then the condition $(\AL1)$ can not be verified.
    \item Let $l\defeq  \mathfrak{r} $ with basis $\{x,y,z\}$. We have $$\ad_x=\begin{pmatrix} 0 & 0 & 0\\ 0&1&1\\0&0&1 \end{pmatrix}\text{; }
    \ad_y= \begin{pmatrix}0&0&0 \\ -1 &0&0\\0&0&0 \end{pmatrix} \text{ and } \ad_z=\begin{pmatrix} 0&0&0\\-1&0&0\\-1&0&0\end{pmatrix}.$$ Then we have 
   $$(\ad_x)^p=\begin{pmatrix} 0 & 0 & 0\\ 0&1&p\\0&0&1\end{pmatrix}=\begin{pmatrix} 0 & 0 & 0\\ 0&1&0\\0&0&1\end{pmatrix}.$$ Hence $(\ad_x)^p$ is not a linear combination of $\ad_x$, $\ad_y$ and $\ad_z$, so we conclude that $\mathfrak{r}$ is not restrictable.
    \item For the end let $t\in k$ and let us have a look at the Lie algebra $\mathfrak{l}_t$ with basis $\{x,y,z\}$.  We have $$\ad_x=\begin{pmatrix} 0 & 0 & 0\\ 0&1&0\\0&0&t \end{pmatrix} \text{ ; }
    \ad_y= \begin{pmatrix}0&0&0& \\ -1 &0&0\\0&0&0 \end{pmatrix} \text{ and } \ad_z=\begin{pmatrix} 0&0&0\\0&0&0\\-t&0&0\end{pmatrix}.$$ Then we have 
   $$(\ad_x)^p=\begin{pmatrix} 0 & 0 & 0\\ 0&1&0\\0&0&t^p\end{pmatrix} \text{ and } (\ad_y)^p=(\ad_z)^p\equiv 0.$$ Then using Theorem~\ref{Jacobson} (Jacobson's theorem), and the definition of a restrictable Lie algebra,
   we know that $l$ is restrictable if and only if $t^p=t$ i.e. if and only if $t\in\mathbb{F}_p$. \hfill $\square$
\end{enumerate}

\begin{example}
Thanks to this classification, we can illustrate
Theorem~\ref{Yrpz}. Indeed, let $k$ be an algebraically
closed field of characteristic $p>0$. Let $\mathfrak{l}_T\defeq
\Spec(k[T])\rightarrow \rL_3$, given on the rings by
$a,c,d,e,g,h,i \mapsto 0$, $b\mapsto 1$ and $f\mapsto
T$. Let us calculate $\rL_3^{\res}\times {\mathfrak{l}_T}$. Thanks to what we
have done before, we know $$\rL_3^{\res}\times {\mathfrak{l}_T}=\Spec
k[T]/(T^p-T).$$ Then we see that $\rL_3^{\res}\times {\mathfrak{l}_T}$ is closed in
$\Spec(k[T])$.
\end{example}

\begin{notitle}{First consequences for the topology of $\rL_{3,k}$}
\end{notitle}
To finish this subsection, we derive the first topological description of the irreducible components of the moduli space
that the classification just given affords. Finer information
can only be obtained with the more advanced algebraic tools
of liaison theory presented in Subsection~\ref{algebraic_treatment}.
First, note that:
\begin{itemize}
\item the points corresponding to the Lie algebras $\fab_3$
and $\fh_3$ are in the closure of the orbit of the simple
algebra $\fs$;
\item the point corresponding to the Lie algebras $\mathfrak{r}$
is in the closure of the orbit of the 1-parameter algebra
$\mathfrak{l}_T$ (to see this, let $k$ be a field and let $t\in k$, and consider the Lie algebra defined by
the structure matrix
$\left(\begin{smallmatrix} 0&0&0\\ 1&1&0\\0&t&0 \end{smallmatrix}\right)$.
For $t\ne 1$, the structure constants of this algebra in the
basis $\{x,y,y+(t-1)z\}$ are those of $\mathfrak{l}_t$ and when $t\to 1$
the limit of this family is $\mathfrak{r}$).
\end{itemize}
Therefore, in order to single out the irreducible components
of $\rL_{3,k}$ it is enough to look at~$o(\fs)$ and $o(\mathfrak{l}_T)$.
We consider their orbit morphisms :
\[
\ev_{\fs}:\GL_3\times\Spec(\Z)\longrightarrow \rL_{3}
\quad,\quad
\ev_{\mathfrak{l}_T}:\GL_3\times\Spec(\Z[T])\longrightarrow \rL_{3}.
\]
We obtain the following result.

\begin{lemma}\label{topologyirr}
In each geometric fiber over a point $\Spec(k)\to\Spec(\Z)$, the following hold:
$\ev_{\fs}$ and $\ev_{\mathfrak{l}_T}$ have 6-dimensional image,
their sum $\GL_{3,k}\amalg\GL_{3,k[T]}\to \rL_{3,k}$ is dominant,
and $\rL_{3,k}$ has pure dimension 6 with two irreducible
components.
\end{lemma}

\begin{proof}
Everything takes place in the fiber over $\Spec(k)\to\Spec(\Z)$
so for simplicity we omit~$k$ from the notation. The stabilizer
of $\fs$ has
dimension~3, hence its orbit (the image of $\ev_{\fs}$) has
dimension~6. For the orbit of $\mathfrak{l}_T$ we may as well remove the
value $t=1$ without changing the dimension. Then the stabilizer of
$\mathfrak{l}_T$ is flat, of dimension~2 over $\Spec(k[T,(T-1)^{-1}]$, hence it has
dimension~3 over $k$, and again the orbit (the image of
$\ev_{\mathfrak{l}_T}$) has dimension~6. The fact that
$\GL_{3,k}\amalg\GL_{3,k[T]}\to \rL_{3,k}$ is dominant follows
from the fact that the remaining orbits lie in the closure of
those two, as we indicated before the lemma. Finally since both
images of $\ev_{\fs}$ and $\ev_{\mathfrak{l}_T}$ are distinct, irreducible,
of dimension~6, their closures are the irreducible components
of~$\rL_{3,k}$.
\end{proof}

\subsection{Schematic description of the moduli space \texorpdfstring{$\rL_3$}{}}\label{algebraic_treatment}\label{subsection4.2}

Let us now focus on the schematic structure of
the moduli space of three-dimensional Lie algebras. We first
prove the representability of the functor $\rL_3$ over the ring of integers.

\begin{definition}\label{D16}
The \textit{moduli space of based Lie algebras of rank three} is the following functor: 
\begin{align*}
    \rL_3:  \Sch &\longrightarrow \Set\\
     T &\longmapsto   \big\{[\cdot,\cdot]: \mathcal{O}_T^3\otimes \mathcal{O}_T^3\rightarrow \mathcal{O}_T^3 \text{ ; where } [\cdot,\cdot] \text{ is a Lie bracket}\big\}.
\end{align*}
\end{definition}


\begin{proposition}\label{P7}
This functor is representable by a closed subscheme of $\mathbb{A}_\Z^9$, given by
$$\displaystyle\Spec\big(\Z[a,b,c,d,e,f,g,h,i]\big/\langle ah+di-fg-bg,ie+bd-fh-ae,hc+dc-af-bi\rangle\big).$$
\end{proposition}

\begin{proof}
Let $T$ be a scheme, and let $\{x,y,z\}$ be a
$\mathcal{O}_T(T)$-basis of $\mathcal{O}_T(T)^3$.\\
Let us write $(a,b,c,d,e,f,g,h,i)~\in~\mathcal{O}_
T(T)^9$ for the coefficients of the Lie bracket
$[\cdot,\cdot]$, where $$[x,y]=ax+by+cz \text{ ,
}[x,z]=dx+ey+fz \text{
and } [y,z]=gx+hy+iz.$$ Then by definition, we
have: 
\begin{align*}
\rL_3(T)
&=\big\{[\cdot,\cdot]: \mathcal{O}_T^3\times \mathcal{O}_T^3\rightarrow \mathcal{O}_T^3 \text{ , where } [\cdot,\cdot] \text{ is a Lie bracket}\big\}\\
& \simeq \big\{(a,\dots,i)\in \mathcal{O}_T(T)^9 \text{ ; } ah+di-fg-bg=ie+bd-fh-ae= hc+dc-af-bi=0\big\}.
\end{align*}
One can easily verify that the conditions on the $9$-tuple correspond to the Jacobi condition.
\end{proof}

\noindent \textbf{Notations: }From now on, we will use the following notations:
\begin{itemize}
   \item[-] $Q_1\defeq ah+di-fg-bg\text{, }Q_2\defeq ie+bd-fh-ae \text{ and } Q_3\defeq hc+dc-af-bi$ and 
 $R_3\defeq \Z[a,b,c,d,e,f,g,h,i]\big/(Q_1,Q_2,Q_3)$, hence $\rL_3=\Spec(R_3)$. For any ring $A$, we write $R_{3,A}$ for $R_3\otimes A$.
    \item[-] Let us remark that the Jacobi condition can be written as
$$\left\{\begin{array}{rcl}
Q_1=ah+di-fg-bg=0 \\Q_2=ie+bd-fh-ae=0 \\Q_3=hc+dc-af-bi=0
\end{array}
\right. \Leftrightarrow 
\left \{
\begin{array}{rcl}
(a-i)h+(b+f)(-g)+(d+h)i&=&0\\
(a-i)(-e)+(b+f)(-h)+(d+h)b&=&0\\
(a-i)(-f)+(b+f)(-i)+(d+h)c&=&0.
\end{array}
\right.$$

Then let us denote $$M \defeq\begin{pmatrix}
h & -g&i\\ -e&-h&b\\-f&-i&c\end{pmatrix} \text{ and } X\defeq\begin{pmatrix} L_1\defeq a-i\\L_2\defeq b+f\\ L_3\defeq d+h \end{pmatrix}.$$ 

Then, the Jacobi condition is verified if and only if $M X=0$.

Actually here, we find again the underlying set of the
irreducible components of $\rL_3$ that we saw in Lemma \ref{topologyirr}. Indeed,
$MX=0$ if and only if $X=0$ or $X\neq 0$ so the matrix $M$ is not injective,
i.e. $\det(M)=0$ and there exists a nonzero vector in its kernel.
    \item[-] For these reasons, we finally set $$L\defeq (L_1,L_2,L_3) 
    \text{, }I\defeq (Q_1,Q_2,Q_3) \text{ and }
    J=(Q_1,Q_2,Q_3,\det(M))=I+(\det(M)),$$ and we will see that the two
    irreducible components are given, as schemes, by the ideals $L$ and $J$, and we 
    will give a more precise description of them. When it is clear from the context, we will still write $I,J$ and $L$ for those ideals seen in $R_{3,A}$ for any ring $A$.
\end{itemize}

\begin{notitle}{Description of the irreducible components}
\end{notitle}

\begin{theorem}
The affine scheme $\rL_3$ can be decomposed in two irreducible components: the
first one is $$\rL_3^{(1)}\defeq
\Spec\big(\Z[a,\dots,i]/L \big) \simeq \mathbb{A}^6$$ and
the second one is $$\rL_3^{(2)}\defeq
\displaystyle\Spec\big(\Z[a,\dots,i]/ J \big).$$
These irreducible components are linked to each other, they are both
Cohen-Macaulay, flat over $\Z$ with integral geometric fibers of dimension
$6$.
\end{theorem}
Let $A$ be any regular ring (for the following we will use $A=\Z$, $A=\Q$ or $A=\mathbb{F}_p$). We have~$R_{3,A}/\left(L\otimes A\right)\simeq A[a,b,c,d,e,g]$ so $L\otimes A$ is prime in
$\rL_{3,A}$. Let us show that the ideal $L$ describes an irreducible component of $\rL_{3,A}$. Let us denote $D\defeq A[a,\dots i]$.

\begin{lemma}\label{minimalprime}
The ideal $L$ is minimal in $D$ among the prime ideals containing $I$.
\end{lemma}

\begin{proof}
Let $\textbf{p} \in \Spec(D)$ be such that $I\subset \textbf{p} \subset L$. First of all, because we have $$M\begin{pmatrix} L_1\\L_2\\L_3 \end{pmatrix}=\begin{pmatrix} Q_1\\Q_2\\Q_3 \end{pmatrix},$$ we obtain
$$\det(M)\begin{pmatrix} L_1\\L_2\\L_3 \end{pmatrix}=(\com(M))^\rt\begin{pmatrix} Q_1\\Q_2\\Q_3 \end{pmatrix}.$$
So we can write $$ (\ast) \left\{
    \begin{array}{lll} \det(M)L_1=(-hc+bi)Q_1+(gc-i^2)Q_2+(-gb+ih)Q_3 \\
\det(M)L_2=(ec-bf)Q_1+(hc+if)Q_2+(-ie-bh)Q_3 \\
\det(M)L_3=(ei-hf)Q_1+(fg+ih)Q_2+(-h^2-ge)Q_3\\
\end{array}.
\right.$$
But $\det(M)=-ch^2+gbf+ei^2-gec+hbi-ihf$ then $\det(M) \notin L$, so $\det(M) \notin \textbf{p}$. Thanks to $(\ast)$, this means that $L_1,L_2,L_3 \in \textbf{p}$, i.e. $\textbf{p}=L$. So $L$ is a minimal prime among the prime ideals containing $I$.
\end{proof}

So now we need to show that $J$ also describes schematically an irreducible component of
$\rL_{3,A}$. In order to do this, we use liaison theory.

\begin{definition}
Let $J$ and $L$ be two ideals in a ring $R$. We say that $J$ and $L$ are \textit{linked} in $R$ by an ideal $I$ if $L=[I:J]$ and $J=[I:L]$.
\end{definition}

\begin{lemma}\label{regular}
The sequences $(L_1,L_2,L_3)$ and $(Q_1,Q_2,Q_3)$ are regular in $A$. 
\end{lemma}

\begin{proof}
It is trivial for $(L_1,L_2,L_3)$. For $(Q_1,Q_2,Q_3)$, let us remark that for any ring $R$, any polynomial in $R[X]$ whose leading coefficient is regular, is regular.
But, the variable $g$ appears only in $Q_1$, the variable $e$ appears only in $Q_2$ and the variable $c$ appears only in $Q_3$. Moreover, all of them appear with a regular coefficient.  
Then let us set $C\defeq B[a,b,d,f,h,i]$. Then,
\begin{itemize}
    \item $Q_1 \in C[c,e][g]$ seen as a polynomial in $g$ has a regular leading coefficient, hence is regular
    \item $Q_2 \in (C[g]/(Q_1))[c,e]$ seen as a polynomial in $e$ has a regular leading coefficient, hence is regular
    \item $Q_3 \in (C[g,e]/(Q_1,Q_2))[c]$ seen as a polynomial in $c$ has a regular leading coefficient, hence is regular.
\end{itemize}
\end{proof}

\begin{corollary}\label{koszulcomplexmap} 
Let us denote by $M^\rt$ the multiplication by the matrix $M^\rt:D^3\rightarrow D^3$. 
The two regular sequences $\{L_1,L_2,L_3\}$ and $\{Q_1,Q_2,Q_3\}$ define two Koszul complexes and we have a morphism of Koszul complexes between them:

\begin{center}

\begin{tikzpicture}
  \matrix (m) [matrix of math nodes, row sep=3em, column sep=3em]
    { 0 & D  &  D[2]^3  & D[4]^3  & D[6] & 0 \\
      0 & D[3]  & D[4]^3  & D[5]^3  & D[6] & 0 . \\ };
  { [start chain] \chainin (m-1-1);
    \chainin (m-1-2);
    { [start branch=A] \chainin (m-2-2)
        [join={node[right] {$\det(M^\rt)$}}];}
    \chainin (m-1-3) [join={node[above] {$d^Q_3$}}];
    { [start branch=B] \chainin (m-2-3)
        [join={node[right] {$\wedge^2M^\rt$}}];}
    \chainin (m-1-4) [join={node[above] {$d^Q_2$}}];
    { [start branch=C] \chainin (m-2-4)
        [join={node[right] {$M^\rt$}}];}
    \chainin (m-1-5) [join={node[above] {$d^Q_1$}}] ; 
    { [start branch=D] \chainin (m-2-5)
        [join={node[right] {$\id$}}];}
        \chainin (m-1-6) ;
        }
  { [start chain] \chainin (m-2-1);
    \chainin (m-2-2);
    \chainin (m-2-3) [join={node[above] {$d^L_3$}}];
    \chainin (m-2-4) [join={node[above] {$d^L_2$}}];
    \chainin (m-2-5)[join={node[above] {$d^L_1$}}];
    \chainin (m-2-6);}
\end{tikzpicture}.
\end{center}
Here: 
$D[n]$ is the graded ring $D$ where we shift the graduation $n$ times, in order to have a morphism of graded
rings (i.e. a polynomial of degree $d$ in $D$ is seen
in the ($d-n$)-th graduation of $D[n]$). 

\end{corollary}

\begin{proof}
This diagram comes from the definition of the Koszul complex (see for example Eisenbud's book \cite{E95}, Section~17, Subsection 17.2) and the functoriality of the Koszul complex: indeed we have $\wedge^0(D^3)~=~\wedge^3(D^3)=D$ and $\wedge^1(D^3)=\wedge^2(D^3)=D^3$, and the morphism $\id$ is just the morphism $\wedge^0M^\rt$, the morphism $M^\rt$ is $\wedge^1M^\rt$ and the morphism $\det(M^\rt)$ is $\wedge^3M^\rt$.
\end{proof}

\begin{remark}
In the following, we will not need the graduation of our complex, so we
will be writing and using it without specifying the graduation.
\end{remark}

\begin{corollary}\label{resolution}
A projective resolution of $[I:L]$ can be obtained by taking the mapping cone of the map of Koszul complexes $(M^\rt)^\vee: K\textbf{.}[L_1,L_2,L_3]^\vee \rightarrow K\textbf{.}[Q_1,Q_2,Q_3]^\vee$.
\end{corollary}

\begin{proof}
This is straightforward from Proposition~2.6 in \cite{PS74}.
\end{proof}

\begin{corollary}\label{perfect}
The ideal $[I:L]$ is perfect of height $3$. Moreover, 
$$[I:L]=I+\det(M)=J.$$
\end{corollary}

\begin{proof} Using the notations of Corollary~\ref{koszulcomplexmap}, we can see that the mapping cone of $M^\rt$ is the following complex: 
$$0 \rightarrow D \oplus 0 \xrightarrow{u} D^3\oplus D \rightarrow D^3\oplus D^3\xrightarrow{} D \oplus D^3 \xrightarrow{f} 0\oplus D \rightarrow 0 $$

\noindent where the morphism $f$ is defined by:
\begin{align*}
    f: D \oplus D^3 &\rightarrow D\\
    (a,b,c,d) &\mapsto a+bL_1+cL_2+dL_3.
\end{align*} 

\noindent Then, dualizing the complex, we obtain: 
$$0 \rightarrow (D)^\vee \xrightarrow{f^\vee} (D \oplus D^3)^\vee
\rightarrow ( D^3\oplus D^3)^\vee \rightarrow (D^3\oplus D)^\vee \rightarrow
(D )^\vee \rightarrow 0$$

\noindent where the morphism $f^\vee$ is defined by:
\begin{align*}
    f^\vee :(D)^\vee & \rightarrow (D \oplus D^3)^\vee\\
    \phi &\mapsto \big( (a,b,c,d) \mapsto \phi(a+bL_1+cL_2+dL_3)\big).
\end{align*}

\noindent Replacing the morphism $f^\vee$ with its image,
and showing this image is projective, we
manage to reduce the length of this resolution. 
Indeed, let us denote by $H$ the kernel of $f$: $$H\defeq \{(a,b,c,d)\in D^4 \text{, }
a+bL_1+cL_2+dL_3=0\}.$$
Then let us show $$\im(f^\vee)=(D^4/H)^\vee\simeq (D^3)^\vee=\{\psi \in (D^4)^\vee
\text{, }\psi_{|H} \equiv 0\}.$$ Let $\psi$ be a form on $D^4$ such
that $\psi(H)=0$. Then for all $b,c,d \in D$, $$\psi(-bL_1-cL_2-dL_3,b,c,d)=0.$$ Let $(a,b,c,d)
\in D^4$. Then
\begin{align*}
    \psi(a,b,c,d)&=\psi(a+bL_1+cL_2+dL_3,0,0,0)+\psi(-bL_1-cL_2-dL_3,b,c,d)\\
    &=\psi(a+bL_1+cL_2+dL_3,0,0,0).
\end{align*}
Let us set $\phi: D\rightarrow D$, $x\mapsto \psi(x,0,0,0)$. Then we obtain
$\psi=f^\vee(\phi)$. The other inclusion is trivial.

Then, $\im(f)$ is free over $D$, so the projective resolution given by
Corollary~\ref{resolution} can be changed into this one: 
$$0 \rightarrow \im(f^\vee) \rightarrow ( D^3\oplus D^3)^\vee \rightarrow
(D^3\oplus D)^\vee \rightarrow (D )^\vee \rightarrow 0$$
which is a projective resolution of length $3$ of $D/[I:L]$. Then
$\projdim([I:L])~\leq~3$. But because $I\subset [I:L]$, we know that
$\grade([I:L])\geq 3$. Hence $[I:L]$ is perfect of grade $3$. 

Now in order to show $[I:L]=I+\det(M)$, we will see that the resolution
found above is actually a resolution of $I+\det(M)$. Let us calculate the
cokernel of the dual of this map:
\begin{align*} u: D& \rightarrow D^3\oplus D\\
1 & \mapsto (-Q_3,Q_2,-Q_1,\det(M)).
\end{align*}
Then the dual map is given by 
\begin{align*} u^\vee: (D^3 \oplus D)^\vee  & \rightarrow  D^\vee\\
\phi & \mapsto (1\mapsto \phi(-Q_3,Q_2,-Q_1,\det(M))). 
\end{align*}
Then the cokernel of this morphism is given by $D/(I+\det(M))$.
So by uniqueness of the cokernel, we have $[I:L]=I + \det(M)=J$.
\end{proof}

In order to show that the ideals $L$ and $J$ are
linked, it remains to show that $L=[I:J]$. It is
one of the purposes of the following proposition,
which is the main result of liaison theory that we
will be using in this article. It will give us
powerful tools to understand the ideal $J$ thanks
to the ideal $L$. For more convenience, let us
denote  by $\underline{L}\defeq L/I$ and
$\underline{J}\defeq J/I$ the two quotient ideals
in the quotient ring $R_{3,A}$. We have seen that
$\underline{J}$ is the annihilator
of $\underline{L}$ in $R_{3,A}$. 

\begin{proposition}\label{asso}
The ideal $\underline{L}$ is the annihilator of $\underline{J}$, and $R_{3,A}/\underline{J}$ has Cohen-Macaulay geometric fibers of dimension $6$.
\end{proposition}

\begin{proof}
Let $\textbf{p} \in \Spec(R_{3,A})$ and let us
denote $R\defeq (R_{3,A})_{\textbf{p}}$. Let us
denote $I_1\defeq (\underline{L})_\textbf{p}$ and
$I_2\defeq (\underline{J})_\textbf{p}$. We would like to apply Proposition~1.3 from \cite{PS74}. Let us show that we are in good conditions:
\begin{enumerate}
\item $R$ is a Gorenstein local ring: indeed, $A[a,\dots,i]$ is regular hence Gorenstein, but because $I=(Q_1,Q_2,Q_3)$ is a regular sequence, then $A[a,\dots,i]/I$
   is also Gorenstein. Hence $R$ is Gorenstein as a localisation of a Gorenstein ring.
   \item $I_2=\ann(I_1)$: indeed $I_2=(\underline{J})_\textbf{p}=(\ann(\underline{L}))_\textbf{p}=\ann((\underline{L})_\textbf{p})$ because $R$ is Noetherian.
   \item $\dim(R)=\dim(R/I_1)$ because as $R$ is Gorenstein hence Cohen-Macaulay, so we can apply Proposition~2.15 d) in Chapter~8, Section~8.2.2 in Liu's book \cite{L02}, using $I_1$ as prime ideal which has height $0$ (because $L$ is a minimal prime thanks to Lemma~\ref{minimalprime}).
   \item  $R/I_1$ is regular hence Cohen-Macaulay. Then using Proposition~1.3 in \cite{PS74}, we obtain that
$$(\underline{L})_\mathbf{p}=[0:(\underline{J})_\mathbf{p}]=[0: \underline{J}]_\mathbf{p}.$$
Because we obtain this result for all $p\in \Spec(R)$ and because we
already know the inclusion $\underline{L} \subset [0:\underline{J}]$,
then we have the equality not only
locally but globally
$$\underline{L}=[0:\underline{J}].$$ Hence $L$ is an associate prime of $I$, and $L$ is the
annihilator of $J$ in $R_{3,A}$. 
\end{enumerate}
The end of the lemma follows from Proposition~1.3
in \cite{PS74}.
\end{proof}

Then now we know that the ideals $J$ and $L$ are
linked, and thanks to this the previous proposition
says that because $D/L$ is Cohen-Macaulay, then $D/J$
is Cohen-Macaulay as well. Thanks to this, we will
prove that, for any algebraically closed field $k$, the ideal $J$ is prime in $R_{3,k}$, then it
describes schematically the second irreducible
component of $\rL_{3,k}$. We need this preliminary
lemma first.

\begin{lemma}\label{smoothpoint}
Let $k$ be a field.
The scheme $\rL_{3,k}^{(2)}=\Spec(k[a,\dots,i]/J)$ has a smooth point.
\end{lemma}

\begin{proof}
Actually, we will find a $\Z$-point of $\rL_3$ along
which $\rL_3$ is smooth. Let $t \in \Z$ and let
$\mathfrak{l}_t: \Spec(\Z)\rightarrow \Spec(\Z[a,\dots,i]/J) $ given
on the rings by $a,c,d,e,g,h,i\mapsto 0 ,  b\mapsto
1$ and $f\mapsto t$. Let us recall that $\det(M)=-ch^2+gbf+ei^2-gec+hbi-ihf$. 

\noindent Then, $$\mathfrak{l}_t^*\big(\Omega^1_{\rL_{3}^{(2)}}\big)=\mathfrak{l}_t
^*\bigg(\frac{\Z\cdot \rd a\oplus \dots \oplus \Z\cdot \rd i}{\rd
Q_1,\rd Q_2,\rd Q_3,\rd \det(M)}\bigg)$$ and we have the following
equalities:

\begin{align*}
    \rd Q_1&=a(\rd h)+h(\rd a)+d(\rd i)+i(\rd d)-f(\rd g)-g(\rd
    f)-b(\rd g)-g(\rd b)\\
    \rd Q_2&=i(\rd e)+e(\rd i)+b(\rd d)+d(\rd b)-f(\rd h)-h(\rd
    f)-a(\rd e)-e(\rd a)\\
    \rd Q_3&=h(\rd c)+c(\rd h)+d(\rd c)+c(\rd d)-a(\rd f)-f(\rd
    a)-b(\rd i)-i(\rd b)\\
    \rd \det(M)&=(fg+hi)\rd b+(-eg-h^2)\rd c+(-cg+i^2)\rd
    e+(bg-hi)\rd f\\
    &\hspace{0.5cm}+(-ce+bf)\rd g+(-2ch+bi-fi)\rd h+(bh-fh+2ei)\rd
    i.
\end{align*}

\noindent Then, $$\mathfrak{l}_t^*\big(\Omega^1_{\rL^{(2)}_{3}}\big)=\frac
{\Z\cdot \rd a\oplus \dots \oplus \Z\cdot \rd i}{(t+1)\rd g,\rd
d-t \rd h,\rd i+t \rd f,t \rd g}=\frac{\Z\cdot \rd a\oplus \dots
\oplus \Z\cdot \rd i}{\rd g,\rd d-t \rd h,\rd i+t \rd f}$$ so it is
a free $\Z$-module of rank $6$. But thanks to Proposition
\ref{asso}, we know that $\rL_3^{(2)}$ has dimension $6$, thus this
$\Z$-point is smooth and the proof is done.
\end{proof}

\begin{corollary}
The ideal $J$ is prime in $R_{3,k}$ for any algebraically closed field $k$, and the scheme $\Spec(k[a,\cdots,i]/J)$ is integral.
\end{corollary}

\begin{proof}
Let $k$ be an algebraically closed field.
We have proved in the Lemma~\ref{topologyirr} that
$V(J)~=~\overline{o(\mathfrak{l}_T)} \subset R_{3,k}$ which is irreducible. Moreover, we saw in Lemma~\ref{smoothpoint} that it
has a smooth point. But because we know from
Proposition~\ref{asso} that it is Cohen-Macaulay, then without associated points, 
so because it is generically reduced, it is reduced. Hence $V(J)$ is integral and $J$ is prime.
\end{proof}

\begin{proposition}\label{intersection}
In the polynomial ring $A[a,\dots,i]$, we have the equality $$I=J\cap L.$$
\end{proposition}

\begin{proof}
Thanks to Corollary~3.5 in Section~3, Subsection 3.2 of \cite{E95}, it is sufficient to show that $(J\cap L)_\textbf{p}\subset I_\textbf{p}$ for all primes $\textbf{p}$
associated to $I$. Thanks to Proposition~\ref{asso}, we know that $L$ is an associated prime of $I$.

Now let $\textbf{p}$ be such a prime ideal. As $(J\cap L)_\textbf{p} \subset J_\textbf{p}\cap L_\textbf{p}$, it is sufficient to prove that $J_\textbf{p}\cap L_\textbf{p}\subset I_\textbf{p}$. 
\begin{itemize}
    \item If $\textbf{p}=L$, we will show that $I_\textbf{p}=L_\textbf{p}$. Then let $\frac{a}{b}\in L_\textbf{p}$, with $a\in L$ and $b\notin L$. Then, because $\det(M)\notin L$, $\frac{a}{b}=\frac{a\det(M)}{b\det(M)} \in I_\textbf{p}$ because we showed in Lemma \ref{minimalprime} that $\det(M)L_1\in I$. Then $I_\textbf{p}=L_\textbf{p}$, so $J_\textbf{p}\cap L_\textbf{p} \subset I_\textbf{p}$.
    \item If $\textbf{p}\neq L$, we will show that $I_\textbf{p}=J_\textbf{p}$. Let
    $\frac{\det(M)}{b}\in J_\textbf{p}$. As $L\neq \textbf{p}$ and $\textbf{p}$ is
    minimal, we have $L\nsubseteq \textbf{p}$, so we can suppose $L_1\notin
    \textbf{p}$. Then because $\det(M)L_1\in I$ and
    $\frac{\det(M)}{b}~=~\frac{\det(M)L_1}{bL_1}$, we have the equality
    $I_\textbf{p}=J_\textbf{p}$. Hence $J_\textbf{p}\cap L_\textbf{p}\subset
    I_\textbf{p}$.  
\end{itemize}
\end{proof}

\begin{corollary}
The ideal $J$ is minimal among the prime ideals containing $I$.
\end{corollary}

\begin{proof}
Let $\textbf{p} \in \Spec(A[a,\dots,i])$ such that $I\subset \textbf{p}\subset J$. If $\det(M)\in \textbf{p}$, then $J=\textbf{p}$. Otherwise, we have $L\subset \textbf{p} \subset J$, hence $I=J\cap L= L$ which is impossible. Then $\textbf{p}=J$ and $J$ is a minimal prime among the ones containing $I$.
\end{proof}

Now we have our two relative irreducible components denoted by $\rL_{3,A}^{(1)}$ and $\rL_{3,A}^{(2)}$, we still have to prove the flatness of $\rL_3^{(2)}$ over the ring of integer $\Z$. We need a preliminary lemma first.

\begin{lemma}\label{flatcoker}
Let $R$ be any commutative ring with unit, and $d\geq 1$. Let
$$0\rightarrow P_{d+1}\rightarrow P_{d} \rightarrow \dots \rightarrow
P_1\rightarrow M\rightarrow 0$$ be an exact sequence of $R$-modules such
that all $P_1,\dots, P_d$ are $R$-flat and we suppose that this exact
sequence is still exact after any base change $R\rightarrow R/I$ where
$I$ is an ideal of $R$. Then, $M$ is also $R$-flat.
\end{lemma}

\begin{proof}
We do an induction on the integer $d$. If $d=1$, then this is
classic. If $d>1$, let
$$0\rightarrow P_{d+2}\xrightarrow{\phi_{d+2}} P_{d+1} \xrightarrow{\phi_{d+1}}
P_{d}\rightarrow \dots \rightarrow P_1\rightarrow M\rightarrow 0$$ be an exact sequence with $d+3$ terms, which is like in the statement. Let us define $$C\defeq
\coker(\phi_{d+2})=P_{d+1}/\im(\phi_{d+2})=P_{d+1}/\ker(\phi_{d+1})=\im(\phi_{d+1})=\ker(
\phi_{d}).$$ Because the inclusion $P_{d+2}\xrightarrow{\phi_{d+2}} P_{d+1}$
is universally injective by hypothesis, the cokernel $C$ is flat over $R$. Then the following exact sequence:
$$0\rightarrow C \xrightarrow{\phi_d} P_{d} \rightarrow \dots \rightarrow
P_1\rightarrow M\rightarrow 0$$ is an exact sequence of $d+2$ terms which is like in the statement. Then by induction, we conclude that $M$ is $R$-flat.
\end{proof}

\begin{corollary}\label{flat}
The scheme $ \rL_{3}^{(2)}$ is flat over $\Spec(\Z)$.
\end{corollary}

\begin{proof}
Let $B$ be any ring. Then, the resolution found in Corollary~\ref{resolution} is
a resolution of the ideal $J$ in $B[a,\dots,i]$, so it is universal. Then we can apply the previous Lemma~\ref{flatcoker}.
\end{proof}

\begin{proposition}\label{entireflat}
The entire scheme $\rL_3$ is flat over $\Z$, and the ideal $I$ is radical.
\end{proposition}

\begin{proof}
Because we have proved $I=L\cap J$ in Proposition~\ref{intersection}, the following map is injective:
\begin{align*}
    \Z[a,\dots,i]/I&\hookrightarrow \Z[a,\dots,i]/L\times \Z[a,\dots,i]/J\\
    \overline{a}&\mapsto (\overline{a},\overline{a}).
\end{align*}
But both of the rings that appear on the right-hand side are flat over $\Z$, then without torsion, so $\Z[a,\dots,i]/I$ is without $\Z$-torsion, then it is $\Z$-flat.

Moreover, because $L$ and $J$ are both radical and $I=J\cap L$, then $J$ is radical.
\end{proof}


\newpage

\subsection{Summary: a picture of a geometric fiber of our moduli space}\label{picture}
\begin{notitle}{In characteristic $p=2$} \end{notitle}Let $k$ be an algebraically closed
field of characteristic $p=2$. Here is a picture representing the two irreducible
components of $\rL_{3,k}$. The points correspond to the different orbits on it, and we 
specify the restrictable ones. We also write on it the dimension of the
center of those Lie algebras.

\begin{center}
\begin{tikzpicture}[xscale=1.15,yscale=1.3,>=latex]

\path[draw,name path=border1] (0,1) to[] (5.5,-1.5);
\draw[draw,name path=border2] (12,1) to[] (6.5,3.5);
\path[draw,name path=line1] (5.5,-1.5) -- (12,1);
\path[draw,name path=line2] (6.5,3.5) -- (0,1);

\shade[left color=red!40,right color=red!70] 
 (0,1) to[] (5.5,-1.5) -- 
 (12,1) to[] (6.5,3.5) -- cycle;

\path[draw,name path=border3] (-1,-4) to[out=20,in=220] (3,3);
\path[draw,name path=border4] (6,-7) to[out=40,in=210] (9,1);
\path[draw,name path=border5] (-1,-4) to[out=0,in=80] (6,-7);
\path[draw,name path=border6] (3,3) to[out=10,in=140] (9,1);

   \draw[fill=green] (6,1.8,1.9)  circle[radius=2pt] ;
   \fill (6,2.1,1.9)  circle[radius=0pt] node{$\fab_3$, $\mathbf{3}$};
\draw[fill=green] (3.8,2.5,2.2) circle[radius=5pt];
\fill (3.8,3,2.2) circle[radius=0pt] node{$\mathbf{\fh_3}$, $\mathbf{1}$};
\draw[fill=green] (7.5,0.1,0.7) circle[radius=5pt];
\fill (7.5,0.7,0.7) circle[radius=0pt] node{$\mathbf{\mathfrak{l}_{-1}}$, $\mathbf{0}$};
\draw[fill=black](4.5,1.8,1) circle[radius=8pt] ;
\fill (4.5,2.4,1) circle[ radius=0pt] node{$\mathbf{\mathfrak{r} }$, $\mathbf{0}$};

\draw[fill=black] (5.5,-2,2.5) circle[ radius=8pt];
\fill (5,-2.1,0.5) circle[ radius=0pt] node{$\mathbf{\mathfrak{l}_\textit{t}=\mathfrak{l}_{\textit{t}^{-1}}}$, $\mathbf{0}$};
\fill (5,-2.35,0.5) circle[ radius=0pt] node{\scriptsize $ \mathbf{\textit{t}\notin \mathbb{F}_2}$} ;
\draw[fill=green] (3,-1.75,1.8) circle[ radius=8pt];
\fill (3,-1.2,1.8) circle[ radius=0pt] node{$\mathbf{\mathfrak{l}_0}$, $\mathbf{1}$};

\draw[fill=black] (9,1.8) circle[radius=12pt] ;
\fill (9,2.5) circle[radius=0pt] node{$\mathbf{s}$, $\mathbf{0}$} ;

\node at  (10.3,1) [rectangle,draw] (A6) {\LARGE $\mathbf{\mathbb{A}_\textbf{\textit{k}}^6}$};
\node at (7,-4.2,1.6)  [rectangle, draw,] (L32) { \LARGE $\mathbf{\rL_{3,\textbf{\textit{k}}}^{(2)}}$};
  
\shade[top color=black, bottom color=white, opacity=0.1] 
  (-1,-4) to[out=20,in=220] (3,3)  to[out=10,in=140] (9,1)
 to[out=210,in=40] (6,-7) to[out=80,in=0] (-1,-4);

\path[name intersections={of=border3 and line2,by={a}}];
\path[name intersections={of=border4 and line1,by={b}}];

\draw[thick,dashed] (a) to[out=-10,in=130] (b);
\end{tikzpicture}
\end{center}
\vspace{1cm}
\large \underline{\textbf{Caption}}:\normalsize
\begin{minipage}{0.4\linewidth}\hspace{1.5cm}
\begin{itemize}
\item [\textcolor{green}{\textbullet}] Restrictable orbit.
\item [\textcolor{black}{\textbullet}] Non-restrictable orbit
\end{itemize}
\end{minipage}
\hspace{0.5cm}
\begin{minipage}{0.4\linewidth}
\begin{tikzpicture}
\draw (9,1.5)  circle[radius=15.6pt];
\node[draw] at (12.5,1.5) {Orbit of dimension 6};
\draw (9,2.5)  circle[radius=10.4pt];
\node[draw] at (12.5,2.5) {Orbit of dimension 5};
\draw (9,3.5)  circle[radius=6.5pt];
\node[draw] at (12.5,3.5) {Orbit of dimension 3};
\draw (9,4.5)  circle[radius=2.6pt];
\node[draw] at (12.5,4.5) {Orbit of dimension 0};
\end{tikzpicture}
\end{minipage}

\newpage
\begin{notitle}{In characteristic $p\neq2$} \end{notitle}Let $k$ be an algebraically closed field of characteristic $p>2$. As on the previous page, here is a picture representing the two irreducible components of $\rL_{3,k}$ with the orbits on it, and the dimension of the center of those Lie algebras.

\begin{center}
\begin{tikzpicture}[xscale=1.15,yscale=1.3,>=latex]

\path[draw,name path=border1] (0,1) to[] (5.5,-1.5);
\draw[draw,name path=border2] (12,1) to[] (6.5,3.5);
\path[draw,name path=line1] (5.5,-1.5) -- (12,1);
\path[draw,name path=line2] (6.5,3.5) -- (0,1);

\shade[left color=red!40,right color=red!70] 
 (0,1) to[] (5.5,-1.5) -- 
 (12,1) to[] (6.5,3.5) -- cycle;

\path[draw,name path=border3] (-1,-4) to[out=20,in=220] (3,3);
\path[draw,name path=border4] (6,-7) to[out=40,in=210] (9,1);
\path[draw,name path=border5] (-1,-4) to[out=0,in=80] (6,-7);
\path[draw,name path=border6] (3,3) to[out=10,in=140] (9,1);

   \draw[fill=green] (6,1.8,1.9)  circle[radius=2pt] ;
   \fill (6,2.1,1.9)  circle[radius=0pt] node{$\fab_3$, $\mathbf{3}$};
\draw[fill=green] (3.8,2.5,2.2) circle[radius=5pt];
\fill (3.8,3,2.2) circle[radius=0pt] node{$\mathbf{\fh_3}$, $\mathbf{1}$};
\draw[fill=green] (7.5,0.1,0.7) circle[radius=5pt];
\fill (7.5,0.7,0.7) circle[radius=0pt] node{$\mathbf{\mathfrak{l}_{-1}}$, $\mathbf{0}$};
\draw[fill=black] (5,-3.35,1.6) circle[radius=10pt] ;
\fill (5,-2.7,1.6) circle[ radius=0pt] node{$\mathbf{\mathfrak{r} }$, $\mathbf{0}$};
\draw[fill=green] (4,-2,1.6) circle[ radius=5pt] ;
\fill (4,-1.5,1.6) circle[ radius=0pt] node{$\mathbf{\mathfrak{l}_1}$, $\mathbf{0}$};
\draw[fill=black] (7.2,-2.3,2.5) circle[ radius=10pt];
\fill (6.5,-2.2,0.3) circle[ radius=0pt] node{$\mathbf{\mathfrak{l}_\textit{t}=\mathfrak{l}_{\textit{t}^{-1}}}$, $\mathbf{0}$};
\fill (6.5,-2.55,0.3) circle[ radius=0pt] node{\scriptsize $\mathbf{\textit{t}\notin \mathbb{F}_p}$};
\draw[fill=green](2.6,-2.5,2.5) circle[ radius=8pt];
\fill (1.7,-2.5,0.2) circle[ radius=0pt] node{$\mathbf{\mathfrak{l}_\textit{t}=\mathfrak{l}_{\textit{t}^{-1}}}$, $\mathbf{0}$};
\fill (1.7, -2.8, 0.2) circle[radius=0pt] node{\scriptsize $\mathbf{\textit{t} \in \mathbb{F}_p}, \mathbf{\textit{t}\neq 0, \pm1}$};
\draw[fill=green] (2.5,-0.85,1.8) circle[ radius=8pt];
\fill (2.5,-0.3,1.8) circle[ radius=0pt] node{$\mathbf{\mathfrak{l}_0}$, $\mathbf{1}$};

\draw[fill=green] (9,1.8)  circle[radius=12pt] ;
\fill (9,2.5)  circle[radius=0pt] node{$\mathbf{\fs}$, $\mathbf{0}$} ;
  
\node at  (10.3,1) [rectangle,draw] (A6) {\LARGE $\mathbf{\mathbb{A}_\textbf{\textit{k}}^6}$};
\node at (7,-4.2,1.6)  [rectangle, draw,] (L32) { \LARGE $\rm{L}_{\mathbf{3,\textbf{\textit{k}}}}^{\mathbf{(2)}}$};
\shade[top color=black, bottom color=black, opacity=0.1]
  (-1,-4) to[out=20,in=220] (3,3)  to[out=10,in=140] (9,1)
 to[out=210,in=40] (6,-7) to[out=80,in=0] (-1,-4);

\path[name intersections={of=border3 and line2,by={a}}];
\path[name intersections={of=border4 and line1,by={b}}];

\draw[thick,dashed] (a) to[out=-10,in=130] (b);
\end{tikzpicture}
\end{center}
\vspace{1cm}
\large \underline{\textbf{Caption}}:\normalsize
\begin{minipage}{0.4\linewidth}\hspace{1.5cm}
\begin{itemize}
\item [\textcolor{green}{\textbullet}] Restrictable orbit.
\item [\textcolor{black}{\textbullet}] Non-restrictable orbit
\end{itemize}
\end{minipage}
\hspace{0.5cm}
\begin{minipage}{0.4\linewidth}
\begin{tikzpicture}
\draw (9,1.5)  circle[radius=15.6pt];
\node[draw] at (12.5,1.5) {Orbit of dimension 6};
\draw (9,2.5)  circle[radius=10.4pt];
\node[draw] at (12.5,2.5) {Orbit of dimension 5};
\draw (9,3.5)  circle[radius=6.5pt];
\node[draw] at (12.5,3.5) {Orbit of dimension 3};
\draw (9,4.5)  circle[radius=2.6pt];
\node[draw] at (12.5,4.5) {Orbit of dimension 0};
\end{tikzpicture}
\end{minipage}

\newpage

\section{Smoothness of \texorpdfstring{$\rL_3^{\res}$}{} on the flattening stratification of the center}

We did not study all the equations of the singular
locus of $\rL_3$, but using \cite{Macaulay2}, we can
see that the singular locus of $\rL_{3,\Q}^{(2)}$
over $\mathbb{Q}$ is given by an ideal, whose radical
is $I_2(M)+L$, where $I_2(M)$ is the ideal generated by the two-minors of the matrix $M$, the one introduced in
Subsection~\ref{subsection4.2}. In order to study 
the singular locus over $\mathbb{Z}$, we prefer to carry out explicit tangent space computations. 

For the rest of the article, let us denote by
$\mathbb{L}_n\defeq \mathbb{A}^n_{\rL_n}$ the
universal Lie algebra of rank $n$ over
$\rL_n$.
Then in the following, we will study the smoothness of the restricted locus 
$\rL_3^{\res}\hookrightarrow \rL_3$ of the universal Lie algebra $\mathbb{L}_3\rightarrow\rL_3$. As said before, we know from Theorem~\ref{Yrpz} that it is interesting to study it after passing to the
flattening stratification of the center. So for all this section, let us denote $k=\mathbb{F}_p$. All the schemes are understood as $k$-schemes. 

Thanks to the theory of Fitting
ideals (the reader can look at
\spref{0C3C} for more details), we can have an explicit description of the
different strata. 
We write $Z(\mathbb{L}_n)$ the center of the universal Lie algebra.
Let
$$\rL_n=:Z_{-1}\supset Z_0\supset Z_1 \supset \dots $$ be the closed subschemes
defined by the Fitting ideals of $Z(\mathbb{L}_n)$. Then, for $r\geq 0$, let us define
$\rL_{n,r}\defeq Z_{r-1}\setminus Z_r$ is the locally closed subscheme of $\rL_n$ where
$Z(\mathbb{L}_n)$ is locally free of rank $r$. Actually in the following, we will not need
to calculate explicitly the flattening stratification. 
We use the notation $\rL_{n,r}^{\res}$ for the locally~closed subscheme of
$\rL_n$ where the center $Z(\mathbb{L}_n)$ is locally free of rank $r$, and
$\mathbb{L}_n$ is restrictable, i.e. $\rL_{n,r}^{\res}\defeq \rL_{n,r} \cap
\rL_n^{\res}$. 
\subsection{Correspondence between the center of the group and the one of the Lie algebra}

In the following, we will extend the classical equivalence of
categories between locally free Lie $p$-algebras of finite rank with finite locally free group schemes of height $1$, showing that the centers of those
objects correspond to each other. This is remarkable
because the centers are not flat in general. In order
to do this we will use the functor denoted by
$\Spec^*$ in \cite{SGA3}, Tome 1, exposé VII$_A$, $\S
3.1.2$. We need first a preliminary lemma. For all this section, let $S$ be a scheme of characteristic $p>0$.

\begin{lemma}\label{injective} Let $G\rightarrow S$ a 
group scheme. Let $R$ be any ring. Then the following
morphism:
\begin{align*}
    \Lie(G)(R)&\xrightarrow{\exp} G(R[\alpha,\beta]/(\alpha^2,\beta^2))\\
    x&\longmapsto \exp(\alpha \beta x)
\end{align*}
is injective.
\end{lemma}

\begin{proof}
Let us write $R(\alpha,\beta)\defeq R[\alpha,\beta]/(\alpha^2,\beta^2)$ and let $f:\Spec(R(\alpha,\beta))\rightarrow
\Spec(R[\epsilon]/(\epsilon^2))$ be the scheme morphism
coming from this injective ring morphism:
$R[\epsilon]/(\epsilon^2)\hookrightarrow
R(\alpha,\beta)$, $\epsilon \mapsto
\alpha\beta$. Then $f$ is surjective as a topological
map, so because $f^\#$ is injective, $f$ is an
epimorphism in the category of schemes. Then this gives
an injective morphism $G(R(\epsilon))\hookrightarrow G(R(\alpha, \beta))$ which gives by restriction 
to the Lie algebras an injective morphism
$$\Lie(G)(R)\hookrightarrow G(R(\epsilon))\hookrightarrow G(R(\alpha, \beta)).$$
\end{proof}

\begin{proposition}\label{centercorres}
Let $G\rightarrow S$ be a finite locally free group
scheme of height $1$. Let $Z(G)$
denote its center. Then 
$$Z(\Lie(G))=\Lie(Z(G)).$$
\end{proposition}

\begin{proof}
For more convenience,
let us write $\fg\defeq \Lie(G)$ and $\fz\defeq
Z(\fg)$. When $\fl$ is a Lie $p$-algebra, we
use the notation $G_p(\fl)\defeq
\Spec^*(U_p(\fl))$ where $U_p(\fl)$ is the
universal restricted enveloping algebra of
$\fl$ and where the notation $\Spec^*$ comes
from \cite{SGA3}, exposé VII$_A$, $\S 3.1.2.$ and is 
defined for any $S$-scheme
$T\rightarrow S$ by: 
$$G_p(\fl)(T)\defeq \Spec^*(U_p(\fl))(T)=\{x\in
U_p(\fl)\otimes \mathcal{O}_T(T), \epsilon(x)=1
\text{ and } \Delta(x)=x\otimes x\}.$$

Let us show $\fz\subset \Lie(Z(G))$. The inclusion $\fz
\subset \fg$ gives a bialgebra inclusion of universal restricted enveloping algebras
$U_p(\fz)\subset U_p(\fg)$, and looking at the
definition, we see that this gives an inclusion
of functors:
$$G_p(\fz)\subset G_p(\fg)=G$$
where the last equality is because $G$ is of 
height $1$. But actually, this
subfunctor takes its values in the center of 
$G$: indeed, because $\fz$ is an abelian Lie 
algebra, the bialgebra $U_p(\fz)$ is
commutative (because for all $x,y \in \fz$, we 
have $x\otimes y-y\otimes x=[x,y]=0$ in $U_p(\fz)$). Moreover by definition, we have for any $S$-scheme
$T\rightarrow S$, $$G(T)=G_p(\fg)(T)=\{x\in
U_p(\fg)\otimes \mathcal{O}_T(T), \epsilon(x)=1
\text{ and } \Delta(x)=x\otimes x\}$$ where the group 
law of $G(T)$ is given by $(x,y)\mapsto x\otimes y$.
But because the algebra $U_p(\fz)$ is abelian, 
then if $x\in G_p(\fz)(T)$, then $x\in
Z(G_p(\fg))$, i.e. $$G_p(\fz)\subset Z(G).$$
Applying the functor Lie we obtain $$\Lie(G_p(\fz))\subset \Lie(Z(G))$$ but 
looking at \cite{SGA3}, exposé VII$_A$, $\S 3.2.3$, we know that 
$\Lie(G_p(\fz))=\Prim(\mathbb{W}(U_p(\fz)))$ and 
by definition, $\fz\subset \Prim(\mathbb{W}(U_p(\fz)))$ so we have the 
inclusion $$\fz\subset \Lie(Z(G)).$$

Now let us show $\Lie(Z(G))\subset \fz$. 
Let $f:Z(G)\hookrightarrow G$ be the closed immersion. It is a
monomorphism then it is injective on the functor of
points. Let $R$ be any ring and let us denote by 
$R(\alpha,\beta)~\defeq~R[\alpha,\beta]/(\alpha^2,\beta^2)$.
Then we know from \cite{DG70} Chapitre II, $\S
4$, n°3, 3.7 (3), that the following diagram is
commutative: $$  \xymatrix{\Lie(Z(G))(R)
\ar[r]^{\exp} \ar[d]_{\Lie(f_R)} &
Z(G)(R(\alpha,\beta)) \ar[d]^{f_S} \\
    \Lie(G)(R) \ar[r]_\exp & G(R(\alpha,\beta))}$$
and the composed map $\Lie(Z(G))(R)\rightarrow G(R(\alpha,\beta))$ is injective. Moreover, if $x\in \Lie(Z(G))(R)$, then $\exp(\alpha
x)\in Z(G)(R)\subset Z(G)(R(\alpha,\beta))$ hence for all $y\in \Lie(G)(R)$,
$$1=\exp(\alpha x)\exp(\beta y)\exp(-\alpha
x)\exp(-\beta y)=\exp(\alpha \beta [x,y])$$
where the last equality comes from \cite{DG70},
Chapitre II, $\S 4$, n°4, $4.2$ (6), and where
$[x,y]$ is the bracket on $\Lie(G)(R)$. But
$x\mapsto \exp(\alpha\beta x)$ is injective thanks to
Lemma~\ref{injective},
then we obtain $[x,y]=0$ for all $y~\in~\Lie(G)(R)$
then $x\in Z(\Lie(G))(R)$. 
\end{proof}

Thanks to this result, we can count the number of
centerless finite locally free group schemes of order $p^3$ of height $1$ on an algebraically closed field:
\begin{proposition}\label{count} Let $k$ be an algebraically closed field of characteristic $p>0$. 
Up to isomorphism,
\begin{itemize}
    \item[-] if $p=2$, there is only $1$ such group scheme.
    \item[-] if $p\neq 2$, there are $(p+3)/2$ such group schemes.
\end{itemize}
\end{proposition}

\begin{proof}
It suffices to count the centerless restrictable Lie algebras of
rank $3$, classified in \ref{restrictableorbit}. Indeed, because they are centerless, they have only one structure of Lie $p$-algebra so there is only one algebraic group scheme corresponding to it. For
$p\neq 2$, it is useful to remember that, with our
notations, the Lie algebras $\mathfrak{l}_t$ and
$\mathfrak{l}_{t^{-1}}$ are in same orbit when $t\neq 0$.
\end{proof}

This extended equivalence allows us to study the properties
of$\rL_n^{\res}$ to deduce properties on the moduli space of finite
locally free group schemes killed by Frobenius. That is, let $S$ be
a scheme of characteristic $p>0$, and for $r\leq n$, let us recall
the notations $p\text{-}\mathcal{L}ie_{n,r}(S)$ for the category of
$n$-dimensional restrictable $\mathcal{O}_S$-Lie algebras whose
center is locally free of rank $r$, and $\mathcal{G}_{n,r}(S)$ the
category of finite locally free group schemes of order
$p^n$, of height $1$, whose center is
locally free of rank $p^r$. With these notations and using the previous
results, we know that the functor $\Lie$ gives us an equivalence of categories: $$ \Lie:\mathcal{G}_{n,r}(S)\overset{\thicksim}{\longrightarrow}
p\text{-}\mathcal{L}ie_{n,r}(S).$$
Moreover, because $\GL_n$ is smooth, the quotient map $\rL_n\rightarrow \mathcal{L}ie_n$ is smooth, so studying the smoothness of $\rL_n$ is equivalent to study the one of $\mathcal{L}ie_n$. Let us denote by $\rL_n^p\defeq X(\mathbb{L}_n)$ the set of $p$-mappings on $\mathbb{L}_n$. Then the quotient map $\rL_n^p\rightarrow p\text{-}\mathcal{L}ie_n$ is smooth, and
$$\rL^p_n\xrightarrow{\text{Forgetful}}\rL^{\res}_n$$ is an affine
fibration, and if for $r\leq n$, we denote by $\rL^p_{n,r}\defeq \rL^p_n \cap \rL_{n,r}$, we know that $$\rL_{n,r}^p\xrightarrow{\text{Forgetful}}\rL^{\res}_{n,r}$$ is smooth for all $r\leq n$. This is the reason why in the following, we will study the smoothness of $\rL^{\res}_{3,r}$ for $r\leq 3$.
 
\subsection{In the stratum \texorpdfstring{$\rL_{3,0}$}{TEXT}}\label{5.2}
\begin{notitle}{Study of \texorpdfstring{$\rL_{3,0}^{\res}$}{TEXT} in the whole scheme \texorpdfstring{$\rL_3$}{TEXT}} \end{notitle}Thanks to the results we have established before, we can imagine all the $k$-point which are in the orbit of
$\mathfrak{l}_{-1}$ are singular in $\rL_{3,0}^{\res}$ because this orbit is in the
intersection of two irreducible components. Actually, thanks to a
calculation of tangent space, we will see that they are the only singular ones.

\begin{theorem}
If $\car(k)\neq 2$, the singular locus of $\rL_{3,0}^{\res}$ is the orbit of $\mathfrak{l}_{-1}$. If $\car(k)=2$, the scheme $\rL_{3,0}^{\res}$ is smooth.
\end{theorem}

\begin{proof}
\begin{itemize}
    \item[-] If $p\neq 2$. We see in the classification Theorem~\ref{touteslesinfos} that the points of $\rL_{3,0}^{\res}$ are the points which are in the orbit of $\mathfrak{l}_t$ with $t \in \mathbb{F}_p$ and $t \neq 0$, and the points in the orbit of $\fs$. Let us start with $\mathfrak{l}_t$, i.e. let us denote for $t \in \mathbb{F}_p^\ast$ as before the $k$-point $\mathfrak{l}_t \defeq \Spec(k)\rightarrow \rL_{3,k}.$
We need to calculate the local ring of this point. We will show that $\mathcal{O}_{\rL_{3,0}^{\res},\mathfrak{l}_t}=\mathcal{O}_{\rL_3^{\res},\mathfrak{l}_t}$ is smooth.

Let us compute $T_{\rL_3^{\res},\mathfrak{l}_t}$. Let us denote by $N$ the $\mathbb{F}_p[\varepsilon]$-module $\mathbb{F}_p[\varepsilon]x\oplus\mathbb{F}_p[\varepsilon]y \oplus \mathbb{F}_p[\varepsilon]z$. Then we can write $$T_{\rL_3^{\res},\mathfrak{l}_t}=\big\{\text{Structures of Lie algebra on } N \text{, restrictable,  such that }N\otimes \mathbb{F}_p=\mathfrak{l}_t \big\}.$$
Let us use again the matrix notation for a Lie algebra structure over $N$. We denote it by $$\mathfrak{l}_{t,\varepsilon}\defeq \begin{pmatrix} a\varepsilon & d\varepsilon & g\varepsilon \\ 1+b\varepsilon & e\varepsilon & h\varepsilon\\
c\varepsilon & t +f\varepsilon & i\varepsilon\end{pmatrix}.$$

\noindent First of all, $\mathfrak{l}_{t,\varepsilon}$ is a Lie algebra structure if and only if its coefficients satisfy the conditions denoted by $Q_1,Q_2$ and $Q_3$ above, that is if and only if $$\left\{  \begin{array}{ll}
       (1+t)g=0\\
        d-t h=0\\
        ta+i=0.
    \end{array}
\right.$$

\noindent Then, $\mathfrak{l}_{t,\varepsilon}$ is in $T_{\rL_3^{\res}, \mathfrak{l}_t}$ if and only if it is restrictable. In order to see the conditions to be restrictable we will calculate $\ad_x^p$, $\ad_y^p$ and $\ad_z^p$ for any $p$ prime. Let us denote $$\beta\defeq 1+t+\dots + t^{p-1}= \left\{
    \begin{array}{ll}
        0 & \mbox{if } t=1 \\
        1 & \mbox{if } t \neq 1.
    \end{array}
\right.$$

\noindent Now, using the matrix notation in the basis $\{x,y,z\}$, we have 
$\ad_x=\begin{pmatrix} 0 & a\varepsilon & d\varepsilon \\ 
0 & 1 + b\varepsilon & e\varepsilon\\ 
0& c\varepsilon & t+ f\varepsilon\end{pmatrix}$ then for all $p$ prime $\ad_x^p=\begin{pmatrix} 0 & a\varepsilon & d\varepsilon \\ 
0 & 1 & \beta e\varepsilon\\ 
0& \beta c\varepsilon & t \end{pmatrix}.$

\noindent Likewise, $\ad_y=\begin{pmatrix}   -a\varepsilon & 0 & g\varepsilon\\
-1- b\varepsilon & 0 & h\varepsilon \\
-c\varepsilon & 0 & i\varepsilon\end{pmatrix}$ hence
$\ad_y^2=\begin{pmatrix} 0 & 0&0 \\
a\varepsilon & 0 & -g\varepsilon\\ 
0& 0 & 0\end{pmatrix}$
and for all $p>2$,
$\ad_y^p\equiv 0$.

\noindent Likewise, $\ad_z=\begin{pmatrix}  - d\varepsilon & -g\varepsilon & 0\\
-e\varepsilon & -h\varepsilon & 0 \\
-t - f\varepsilon & - i\varepsilon & 0\end{pmatrix}$, hence 
$\ad_z^2=\begin{pmatrix} 0 & 0&0 \\
0& 0 & 0 \\ 
t d\varepsilon & t g\varepsilon & 0 \end{pmatrix}$
and for all $p>2$, $\ad_z^p\equiv 0.$

Then $\ad_x^p$ is a linear combination of $\ad_x$, $\ad_y$ and $\ad_z$ if and only if it exists $\lambda=\lambda_0+\lambda_1 \varepsilon$, $\mu=\mu_0+\mu_1 \varepsilon$ and $\nu=\nu_0+\nu_1 \varepsilon$ such that:
\small
$$\begin{pmatrix} 0 & a \varepsilon & d \varepsilon \\ 
0 & 1 & \beta e\varepsilon\\ 
0& \beta c\varepsilon& t \end{pmatrix} =
\begin{pmatrix}  (-\mu_0 a-\nu_0 d)\varepsilon & (\lambda_0  a -\nu_0g)\varepsilon& (\lambda_0  d+\mu_0g)\varepsilon\\
-\mu_0- (b\mu_0+\mu_1+\nu_0e)\varepsilon & \lambda_0+(b\lambda_0+\lambda_1-\nu_0 h)\varepsilon & (\lambda_0 e+\mu_0h) \varepsilon\\ -t\nu_0- (\mu_0c+s\nu_0+t\nu_1)\varepsilon&  (\lambda_0 c -\nu_0 i)\varepsilon &t\lambda_0+(f\lambda_0+t\lambda_1+\mu_0i)\varepsilon \end{pmatrix}.$$
\normalsize
Then $\ad_x^p$ is a linear combination of $\ad_x$, $\ad_y$ and $\ad_z$ if and only if
$$\left\{  \begin{array}{ll}
        bt=f\\
        c=\beta c\\
        e=\beta e.
    \end{array}
\right.$$

Because $\ad_y^p\equiv \ad_z^p\equiv 0$, they are always linear combination of $\ad_x$, $\ad_y$ and $\ad_z$.

Hence we obtain the following conditions: 
$$(\ast)\left\{  \begin{array}{ll}
       (1+t) g= d-t h=i+t a=0\\
        bt-f= c-\beta c=e-\beta e=0.
    \end{array}
\right.$$

So we have to distinguish different cases. First let us suppose $t=-1$. Then the conditions~$(\ast)$ are equivalent to:

$$\left\{  \begin{array}{ll}
        d+h=0\\
        a-i=0\\
        b+f=0.
    \end{array}
\right.$$

Hence $\dim(T_{\rL_3^{\res},\mathfrak{l}_{-1}})=6$. But $\dim(o(\mathfrak{l}_{-1}))=5$ from Theorem~\ref{touteslesinfos}, hence the local ring of $\mathfrak{l}_{-1}$ is singular.

Let us suppose $t=1$. Then the conditions~$(\ast)$ are equivalent to $$\left\{  \begin{array}{ll}
        g=d- h=a+i=0\\
        b-f=c=e=0
    \end{array}
\right.$$
so $\dim(T_{\rL_3^{\res},\mathfrak{l}_1})=3=\dim(o(\mathfrak{l}_1))$. Then the point $\mathfrak{l}_1$ is smooth.

Then let us suppose $t \neq 1$ and $t \neq -1$. Then the conditions~$(\ast)$ are equivalent to:
$$\left\{  \begin{array}{ll}
        g=0\\
        d-th=i+ta=bt-f=0
    \end{array}
\right.$$

Hence $\dim(T_{\rL_3^{\res},\mathfrak{l}_t})=5$. But $\dim(o(\mathfrak{l}_t))=5$, hence the local ring of $\mathfrak{l}_t$ is regular.

Doing the same calculations for $\fs \in \rL_{3,0}^{\res}$, we obtain these conditions:
$$\left\{  \begin{array}{ll}
        a-i=0\\
        b+f=0\\
        d+h=0.\\
    \end{array}
\right.$$
Hence $\dim(T_{\rL_3^{\res},\fs})=6$. But $\dim(o(\fs))=6$, hence the local ring of $\fs$ is regular.
    \item[-] Let us suppose $p=2$. 
Then the only point in $\rL_{3,0}^{\res}$ is $\mathfrak{l}_1$. So using the same notations as before, we see in this case, the conditions are equivalent to 
$$\left\{  \begin{array}{ll}
       d-h=a+i=b-f=0\\
        c=e=g=0
    \end{array}
\right.$$
so $\dim(T_{\rL_3^{\res},\mathfrak{l}_1})=3$. But $\dim(o(\mathfrak{l}_{1}))=3$, hence the local ring of $\mathfrak{l}_{1}=\mathfrak{l}_{-1}$ is regular.
\end{itemize}

\end{proof}

\begin{notitle}{Study of \texorpdfstring{$\rL_{3,0}^{\res}$}{TEXT} in the first irreducible component}
We start by establishing a result on the scheme structure of $\rL_{3,0}^{\res}$ in the first irreducible component, in the case we choose a field
$k$ of characteristic $p\neq 2$. 
\end{notitle}
\begin{proposition}
The scheme $\rL_{3,0}\cap \rL_3^{(1)}$ is
reduced. Moreover, if $\car(k) \neq 2$, $$\rL_{3,0}^{\res}\cap \rL_3^{(1)}\simeq\rL_{3,0}\cap \rL_3^{(1)} \text{ as schemes.}$$
\end{proposition}

\begin{proof}
Because $\rL_{3,0}$ is open in $\rL_3$, $\rL_{3,0}\cap \rL_3^{(1)}$ is open in the reduced irreducible component $\rL_3^{(1)}$, then it is reduced. Moreover, $$|\rL_{3,0}^{\res}\cap \rL_3^{(1)}|=\bigcup_{R_k\rightarrow k=\overline{k}}(\rL_{3,0}^{\res}\cap \rL_3^{(1)})(k)=\bigcup_{R_k\rightarrow k=\overline{k}}\rL_3^{\res}(k)\times_{S(k)}\rL_{3,0}(k)\times_{S(k)}\rL_3^{(1)}(k)$$ $$=\bigcup_{R_k\rightarrow k=\overline{k}}\rL_{3,0}(k)\times_{S(k)}\rL_3^{(1)}(k)=|\rL_{3,0}\cap \rL_3^{(1)}|.$$ 

Then $\rL_{3,0}^{\res}\cap
\rL_3^{(1)}$ is a closed subscheme of the reduced scheme $\rL_{3,0}\cap \rL_3^{(1)}$ with the same underlying set. Then they are equal as schemes.
\end{proof}

Now we study the $k$-points of this intersection of schemes. We have to do exactly the same calculus we have done in the previous subsection, but we have to change the conditions $Q_1,Q_2$ and $Q_3$ for the conditions $L_1$, $L_2$ and $L_3$. Then we find:

\begin{proposition}
If $\car(k)\neq 2$, in $\rL_{3,0}^{\res}\cap \rL_3^{(1)}$, the $k$-points $\mathfrak{l}_{-1}$ is singular, and $\fs$ is regular. If $\car(k)=2$, the scheme $\rL_{3,0}^{\res}\cap \rL_3^{(1)}$ is smooth.
\end{proposition}

\begin{proof}
\begin{itemize}
    \item[-] Let us suppose $p\neq 2$. We first look at the point $\fs$. We obtain, as before, these conditions:
    $$\left\{  \begin{array}{ll}
        a-i=0\\
        b+f=0\\
        d+h=0.
    \end{array}
\right.$$
Hence $\dim(T_{\rL_3^{\res},\fs})=6$, so the local ring of $\fs$ is regular.

Let us do the same for the point $\mathfrak{l}_{-1}$. Doing the same calculations we obtain $\dim(T_{\rL_3^{\res},\mathfrak{l}_{-1}})=~6$, so the local ring of $\mathfrak{l}_{-1}$ is singular.

    \item[-] If $p=2$, we have $\dim(T_{\rL_3^{\res},\mathfrak{l}_{1}})=3$ so $\mathfrak{l}_{1}$ is regular.
\end{itemize}
\end{proof}

\begin{notitle}{Study of \texorpdfstring{$\rL_{3,0}^{\res}$}{TEXT} in the second irreducible component}
\end{notitle}

\begin{theorem}
In the second irreducible component, all the $k$-points of $\rL_{3,0}^{\res}\cap \rL_3^{(2)}$ are smooth.
\end{theorem}

\begin{proof}
We can do the same proof as before, we just need to add the condition $\det(M)=0$. That is, if we keep the same notations as before, we need to add to the system $(\ast)$ the condition $gt=0$. Hence the new system is given by $$\left\{  \begin{array}{ll}
        g=d-t h=i+t a=0\\
        bt-f=c-\beta c=e-\beta e=0.
    \end{array}
\right.$$

So in this case, we obtain $$
\dim(T_{\rL_3^{\res},\mathfrak{l}_t}) = \left\{
    \begin{array}{ll}
        3 & \mbox{if } t=1 \\
        5 & \mbox{if } t \neq 1.
    \end{array}
\right.
$$                
\end{proof}

\begin{remark}
By a simple computation, we can see that any deformation of Lie algebras which are in the stratum
$\rL_{3,0}$ is centerless without any condition. It is because
the stratum $\rL_{3,0}$ is open in $\rL_3$.
\end{remark}

\subsection{In the stratum  \texorpdfstring{$\rL_{3,1}$}{TEXT}}\label{5.3}

Let us do the same calculations for the points of $\rL_{3,1}$. \begin{notitle}{Study of $\rL_{3,1}^{\res}$ in $\rL_3$}
\end{notitle}
\begin{proposition}\label{h3singular}
The $k$-point $\fh_3$  is singular in $\rL_{3,1}^{\res}$, and $\mathfrak{l}_0$ is smooth.
\end{proposition}

\begin{proof}
\begin{itemize}
    \item[-] For the point $\fh_3$, as in the previous section, let us denote by $\fh_{3,\varepsilon}$ a deformation of the Lie algebra $\fh_3$:
    $$\fh_{3,\varepsilon} \defeq \begin{pmatrix} a\varepsilon & d\varepsilon & 1+g\varepsilon \\ b\varepsilon & e\varepsilon & h\varepsilon\\
c\varepsilon & f\varepsilon & i\varepsilon\end{pmatrix}.$$
Then, $\fh_{3,\varepsilon}$ gives the constants of structure of a Lie algebra if and only if $b+f=0.$ Moreover, $\fh_{3\varepsilon}$ is restrictable if and only if:\begin{itemize}
    \item[-] if $p=2$: $b=c=e=0$
    \item[-] if $p=3$: there is no condition
    \item[-] if $p>3$: there is no condition.
\end{itemize}
For the end, the center $Z(\fh_{3,\varepsilon})$ is locally free of rank $1$ if and only if: \begin{itemize}
    \item if $p=2$: there is no condition
    \item if $p=3$: $b=c=e=0$
    \item if $p>3$: $b=c=e=0$
\end{itemize}

So to conclude we use the fact that $\dim(o(\fh_3))=3$.
    \item[-] For the point $\mathfrak{l}_0$ let us do the same. Then using the same notations, $\mathfrak{l}_{0,\varepsilon}$ gives the constants of structure of a Lie algebra if and only if $d=g=i=0$. Moreover, $\mathfrak{l}_{0,\varepsilon}$ is restrictable if and only if $f=0$.
For the end, the center $Z(\mathfrak{l}_{0,\varepsilon})$ is always locally free of rank $1$.
So we conclude using the fact that $\dim(o(\mathfrak{l}_0))=5$.
 \end{itemize}   
\commentaire{$A$ est dans $\rL_3^{\res}$:

$\ad_x=\begin{pmatrix} 0&a\varepsilon & d\varepsilon \\0& b\varepsilon & e\varepsilon & \\
0&c\varepsilon & f\varepsilon & \end{pmatrix}$ et 
$\ad_y=\begin{pmatrix} -a\varepsilon & 0 & 1+g\varepsilon \\ -b\varepsilon & 0 & h\varepsilon\\
-c\varepsilon & 0 & i\varepsilon\end{pmatrix}$ et $\ad_z=\begin{pmatrix}  -d\varepsilon & -1-g\varepsilon&0\\  -e\varepsilon & -h\varepsilon&0\\
 -f\varepsilon & -i\varepsilon&0\end{pmatrix}$.

Donc $\ad_x^p\equiv 0$ et $\ad_y^2=\begin{pmatrix} -c\varepsilon & 0 & (i-a)\varepsilon \\ 0 & 0 & -b\varepsilon\\
0 & 0 & -c\varepsilon\end{pmatrix}$ BREF}
\end{proof}

\begin{notitle}{Study of $\rL_{3,1}^{\res}$ in $\rL_3^{(1)}$}
\end{notitle}
\begin{proposition}
The $k$-point $\fh_3$ is smooth in $\rL_{3,1}^{\res} \cap \rL_3^{(1)}$.
\end{proposition}

\begin{proof}
We have to add to the conditions found before the conditions $a=i$ and $d=-h$.
\end{proof}

\begin{notitle}{Study of $\rL_{3,1}^{\res}$ in $\rL_3^{(2)}$}
\end{notitle}
\begin{proposition}
The $k$-point $\fh_3$ is singular in $\rL_{3,1}^{\res}\cap \rL_3^{(2)}$ and the point $\mathfrak{l}_0$ is smooth.
\end{proposition}

\begin{proof}
For both of those points, the condition $\det(M)=0$ is always satisfied for any deformation. 
\end{proof}

\subsection{In the stratum \texorpdfstring{$\rL_{3,3}$}{TEXT}}\label{5.4}

This case is really simple because the condition "to
be in the stratum $\rL_{3,3}$" implies, using the
same notations as in the previous subsections, that
all the coefficient of the matrix $\fab_{3,\varepsilon}$ 
are~$0$. Then, $\dim(T_{\rL_{3,3}^{\res},\fab_3})=0$, in
the whole scheme and in the irreducible components.
Then the point $\fab_3$ is smooth seen in 
$\rL_{3,3}^{\res}$.

\bigskip 

As stated in the introduction, we can apply the previous results of smoothness to the moduli space $\mathcal{G}_{3,r}$, and this gives the following result:
\begin{corollary}{\label{corogroup}} Let $k$ be an algebraically closed field of characteristic $p>0$. Then
$\mathcal{G}_{3,r}(k)$ splits in two irreducible components that we denote by $\mathcal{G}_{3,r}^{(1)}$ and $\mathcal{G}_{3,r}^{(2)}$, and we have: 
\begin{itemize}
    \item[-] If $p\neq 2$, $\mathcal{G}_{3,0}(k)$ is singular, but becomes
    smooth after intersection with $\mathcal{G}_{3,0}^{(2)}$, if $p\neq 2$, $\mathcal{G}_{3,0}(k)$ is smooth.
    \item[-] $G^p_{3,1}(k)$ is singular but becomes smooth when we
    intersect it with $\mathcal{G}_{3,1}^{(1)}$.
    \item[-] $\mathcal{G}_{3,2}(k)$ is empty and $G^p_{3,3}(k)$ is smooth.
\end{itemize}
\end{corollary}
\begin{proof}
We have an equivalence of categories given by the functor $\Lie:
\mathcal{G}_{3,r}(k) \rightarrow p\text{-}\mathcal{L}ie_{3,r}$(k), moreover, the quotient
morphism $\rL_{3,r}^p\rightarrow p\text{-}\mathcal{L}ie_{3,r}$ is smooth. Then, because $\rL_{3,r}^p\rightarrow \rL_{3,r}^{\res}$ is smooth, we can
apply the results of the subsections \ref{5.2}, \ref{5.3} and \ref{5.4}.
\end{proof}

\addcontentsline{toc}{section}{Bibliography}

\bigskip

\noindent
Alice Bouillet,
{\sc Univ Rennes, CNRS, IRMAR - UMR 6625, F-35000 Rennes, France} \\
Email address: {\tt alice.bouillet@univ-rennes1.fr}

\end{document}